\documentclass[11pt]{amsart}
\usepackage{amsmath,amssymb}
\usepackage{graphicx}

\usepackage{fullpage}
\usepackage{enumerate}
\usepackage{pgf,tikz}
\usetikzlibrary{arrows}
\usetikzlibrary{positioning}
\newcommand{\msc}[1]{\begin{center}MSC2000: #1.\end{center}}
\def\COMMENT#1{}
\let\COMMENT=\footnote
\def\TASK#1{}

\def\noproof{{\unskip\nobreak\hfill\penalty50\hskip2em\hbox{}\nobreak\hfill%
        $\square$\parfillskip=0pt\finalhyphendemerits=0\par}\goodbreak}
\def\endproof{\noproof\bigskip}
\newdimen\margin   
\def\textno#1&#2\par{%
    \margin=\hsize
    \advance\margin by -4\parindent
           \setbox1=\hbox{\sl#1}%
    \ifdim\wd1 < \margin
       $$\box1\eqno#2$$%
    \else
       \bigbreak
       \hbox to \hsize{\indent$\vcenter{\advance\hsize by -3\parindent
       \sl\noindent#1}\hfil#2$}%
       \bigbreak
    \fi}

\newtheorem{thm}{Theorem}[section]
\newtheorem{define}[thm]{Definition}

\newtheorem{lem}[thm]{Lemma}
\newtheorem{claim}[thm]{Claim}
\newtheorem{fact}[thm]{Fact}

\newtheorem{prop}[thm]{Proposition}

\newtheorem*{thm*}{Theorem}
\newtheorem*{define*}{Definition}
\newtheorem*{examp*}{Example}
\newtheorem*{lem*}{Lemma}
\newtheorem*{claim*}{Claim}
\newtheorem*{fact*}{Fact}
\newtheorem*{col*}{Corollary}
\newtheorem*{conj*}{Conjecture}

\begin{document}

\title{A degree sequence Koml\'{o}s theorem}

\author{Joseph Hyde, Hong Liu, Andrew Treglown}
\thanks{H.L. was supported by the Leverhulme Trust Early Career Fellowship~ECF-2016-523.  A.T.\ was supported by EPSRC grant EP/M016641/1.}
\maketitle
\date{}
\begin{abstract}
An important result of Koml\'os~[Tiling Tur\'an theorems, Combinatorica, 2000] yields the asymptotically exact minimum degree threshold that ensures a graph \(G\) contains an \(H\)-tiling covering an \(x\)th proportion of the vertices of \(G\) (for any fixed \(x \in (0,1)\) and graph \(H\)).
We give a degree sequence strengthening of this result which allows for a large proportion of the vertices in the host graph \(G\) to have degree substantially smaller than that required by Koml\'os' theorem. We also demonstrate that for certain graphs \(H\), the degree sequence condition is essentially best possible in more than one sense.
\end{abstract}

\msc{05C35, 05C70}

\section{Introduction} \label{Introduction}
A central branch of extremal combinatorics concerns the study of conditions that force a graph \(G\) to contain some given substructure. For example, Tur\'an's famous theorem determines the number of edges required to force a graph \(G\) to contain a copy of a fixed clique \(K_r\) on \(r\) vertices. Tutte's theorem characterises all those graphs \(G\) that contain a perfect matching.

The study of \emph{graph tilings} has proven to be a rich topic within this area: given two graphs \(H\)  and \(G\), an \emph{\(H\)-tiling} in \(G\) 
is a collection of vertex-disjoint copies of \(H\) in \(G\). An
\(H\)-tiling is called \emph{perfect} if it covers all the vertices of \(G\).
Perfect \(H\)-tilings are also often referred to as \emph{\(H\)-factors}, \emph{perfect \(H\)-packings} or \emph{perfect \(H\)-matchings}. 
\(H\)-tilings can be viewed as generalisations of both the notion of a matching (which corresponds to the case when \(H\) is a single edge) and the Tur\'an problem (i.e. a copy of \(H\) in \(G\) is simply an \(H\)-tiling of size one).

A cornerstone result in the area is the Hajnal--Szemer\'edi theorem~\cite{hs} from 1970, 
which characterises the minimum degree that ensures a graph contains a perfect \(K_r\)-tiling. 

\begin{thm}[Hajnal and Szemer\'edi~\cite{hs}]\label{hs}
Every graph \(G\) whose order \(n\)
is divisible by \(r\) and whose minimum degree satisfies \(\delta (G) \geq (1-1/r)n\) contains a perfect \(K_r\)-tiling. Moreover, there are \(n\)-vertex graphs \(G\)
 with \(\delta (G) = (1-1/r)n-1\) that do not contain a perfect \(K_r\)-tiling.
\end{thm}

Although the minimum degree condition in the Hajnal--Szemer\'edi theorem is tight, this does not mean one cannot strengthen this result. 
Indeed, Kierstead and Kostochka~\cite{kier} proved an \emph{Ore-type} generalisation of Theorem~\ref{hs} where now one replaces the minimum degree condition with the condition that the sum of the degrees of every pair of non-adjacent vertices in \(G\) is at least
\(2(1-1/r)n-1\). A conjecture of Balogh, Kostochka and Treglown~\cite[Conjecture~7]{bkt} would, if true, give a \emph{degree sequence} strengthening of the Hajnal--Szemer\'edi theorem; in this conjecture one allows for \(G\) to have almost \(n/r\) vertices of degree less than
\((1-1/r)n\). An asymptotic version of this conjecture was proven in~\cite{Tregs}.

There has also been significant interest in the minimum degree threshold that ensures a perfect \(H\)-tiling for an \emph{arbitrary} graph \(H\).
After  earlier work on this topic (see~e.g.~\cite{alonyuster, kssAY}),  
K\"uhn and Osthus~\cite{kuhn, kuhn2}  determined, up to an additive constant, the minimum degree that forces a perfect \(H\)-tiling for \emph{any} fixed graph \(H\).

The focus of this paper is not on \emph{perfect} \(H\)-tilings, but rather on \(H\)-tilings covering an \(x\)th proportion of a graph \(G\) (for some fixed \(x \in (0,1)\)).
The focal result on this topic is a theorem of Koml\'os~\cite{Komlos} which determines asymptotically the minimum degree that ensures a graph \(G\) contains an \(H\)-tiling covering an \(x\)th proportion of its vertices.
Before we can state this result, we require two definitions.
The \emph{critical chromatic number} \(\chi _{cr} (H)\) of a graph \(H\) is defined
as 
\[\chi _{cr} (H):=(\chi (H)-1)\frac{|H|}{|H|-\sigma (H)},\]
where \(\sigma (H)\) denotes the size of the smallest possible 
colour class in any \(\chi(H)\)-colouring of \(H\).
For all \(x \in (0,1)\), define
\[g_H(x):=x\left (1-\frac{1}{\chi _{cr} (H)} \right )+(1-x) \left (1- \frac{1}{r-1} \right ).\]

\begin{thm}[Koml\'os~\cite{Komlos}]\label{kom}
Suppose \(H\) is a graph of chromatic number \(r\). 
Given any \(\eta >0\), there exists an \(n_0=n_0(\eta,x,H)\in \mathbb N\) such that if \(G\) is a graph on \(n \geq n_0\) vertices and
\[\delta (G) \geq g_H(x)n\]
then \(G\) contains an \(H\)-tiling covering at least \((x-\eta)n\) vertices.
\end{thm}
Note that the minimum degree condition in Theorem~\ref{kom} is best possible in the sense that given any fixed \(H\) and \(x \in (0,1)\), one cannot replace \(g_H(x)\) with any fixed \(g'_H(x)<g_H(x)\) (see~\cite[Theorem~7]{Komlos} for a proof of this).
A consequence of the Erd\H{o}s--Stone theorem is that every \(n\)-vertex graph \(G\) with \(\delta (G)\geq (1-1/(\chi(H)-1)+o(1))n\) contains a copy of \(H\). So a way to interpret Theorem~\ref{kom} is that, for very small \(x>0\), the minimum degree threshold
is governed essentially by the value of \(\chi (H)-1\); however, as one increases \(x\), the value of \(\chi_{cr} (H)\) plays an increasing role in the value of the threshold.

An attractive  consequence of Theorem~\ref{kom} is the following result concerning \emph{almost} perfect \(H\)-tilings.
\begin{thm}[Koml\'os~\cite{Komlos}]\label{komcor}
Let \(\eta > 0\) and let \(H\) be a graph. Then there exists an \(n_0=n_0(\eta, H) \in \mathbb{N}\) such that  every graph \(G\) on \(n \geq n_0\) vertices
 with
\[\delta (G) \geq \left ( 1 -\frac{1}{\chi _{cr} (H)} \right )n\]
 contains an \(H\)-tiling covering all but at most \(\eta n\) vertices.
\end{thm}
As with Theorem~\ref{kom}, the minimum degree condition in Theorem~\ref{komcor} is best possible in the sense that one cannot replace the   \((1 -{1}/{\chi _{cr} (H)})\) term here with any smaller fixed constant.
Despite this, Shoukoufandeh and  Zhao~\cite{szhao} proved that one can strengthen the conclusion of the theorem, to ensure the \(H\)-tiling covers all but a \emph{constant} number of vertices (this constant depends only on \(H\)).


The main result of this paper is to prove the following degree sequence strengthening of Theorem~\ref{komcor}.
\begin{thm} \label{almostmain}
Let \(\eta > 0\) and \(H\) be a graph with \(\chi(H) = r\). Let \(\sigma:= \sigma (H)\),  \(h := |H|\) and \(\omega := \left(h - \sigma\right)/(r-1)\). Then there exists an \(n_0 = n_0(\eta, H) \in \mathbb{N}\) such that the following holds: Suppose \(G\) is a graph on \(n\geq n_0\) vertices with degree sequence \(d_1 \leq d_2 \leq \ldots \leq d_n\) such that
\[d_i \geq \left(1 - \frac{\omega+\sigma}{h}\right)n + \frac{\sigma}{\omega}i \ \ \mbox{for all \ \(1 \leq i \leq \frac{\omega n}{h}\).}\]
 Then \(G\) contains an \(H\)-tiling covering all but at most \(\eta n\) vertices.
\end{thm}


Note that if one considers an \(r\)-partition of \(H\) with smallest vertex class of size \(\sigma =\sigma(H)\) and
 set \(i={\omega n}/{h}\) 
 then we obtain that \( \left(1 - ({\omega+\sigma})/{h}\right)n + {\sigma}i/{\omega} = 
 1-1/\chi _{cr} (H) \). Thus, Theorem~\ref{almostmain} is a significant strengthening of Theorem~\ref{komcor}. Indeed, Theorem~\ref{almostmain} allows for up to \(\omega n/h\) vertices to have degree below that in Theorem~\ref{komcor}.
In particular,  when \(H\) is bipartite, the degree sequence condition in Theorem~\ref{almostmain} starts at \(d_1\geq 1\) and allows for at least half of the vertices of \(H\) to have degree less than that required by Koml\'os' theorem.
Figure~\ref{degseqkomthmfigure} gives a visualisation of the degree sequence in Theorem~\ref{almostmain}. Figure~\ref{graphtable} presents some key properties of the degree sequence in Theorem~\ref{almostmain} for several graphs. Here, `Angle of slope' refers to the value \(\sigma/\omega\). 

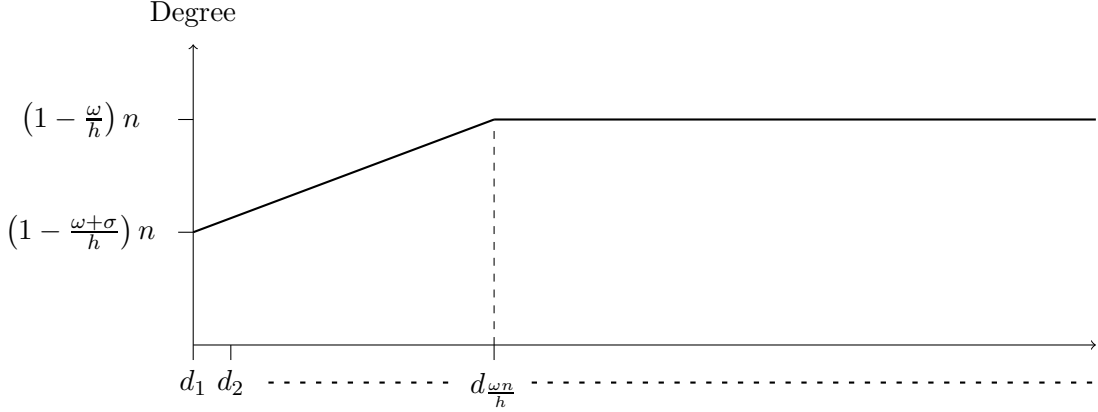
\begin{figure}
\begin{center}
\begin{tikzpicture}\label{degseqkomthmfigure}

\draw[->] (-4,2) -- (-4,6);
\draw[->] (-4,2) -- (8,2);
\draw (-4,3.5) -- (-4.2,3.5);
\draw (0,2) -- (0,1.8);
\draw (-4,5) -- (-4.2,5);
\node at (-4,1.5) {\(d_1\)};
\draw (-4,2) -- (-4,1.8);
\draw (-3.5,2) -- (-3.5,1.8);
\node at (-3.5,1.5) {\(d_2\)};
\draw[dash pattern=on 2pt off 4pt, thick] (-3,1.5) -- (-0.5,1.5);
\node at (0,1.4) {\(d_{\frac{\omega n}{h}}\)};
\draw[dash pattern=on 2pt off 4pt, thick] (0.5,1.5) -- (8,1.5);
\node at (-5.5,3.5) {\(\left(1 - \frac{\omega+\sigma}{h}\right)n\)};
\node at (-5.5,5) {\(\left(1 - \frac{\omega}{h}\right)n\)};
\draw[black,thick] (-4,3.5) -- (0,5);
\draw[black,thick] (0,5) -- (8,5);
\node at (-4,6.4) {Degree};
\draw[dashed] (0,2) -- (0,4.95);
\end{tikzpicture}
\caption{The degree sequence in Theorem~\ref{almostmain}.}
\end{center}
\end{figure}

\begin{figure}
\begin{center}
\begin{tabular}{ |c|c|c|c| } 
 \hline
 Graph          & Bound on \(d_1\)   & Bound on \(d_{\frac{\omega n}{h}}\)      & Angle of slope \\ 
 \hline
 \(C_5\)        & \(2n/5\)           & \(3n/5\)                                 & \(1/2\) \\ 
 \(K_{1,t}\)    & \(1\)              & \(n/(t+1)\)                              & \(1/t\) \\ 
 \(K_t\)        & \((t-2)n/t\)       & \((t-1)n/t\)                             & \(1\)   \\ 
 \(K_{2,4,6}\)  & \(5n/12\)          & \(7n/12\)                                & \(2/5\) \\ 
 \hline
\end{tabular}

\caption{Values of the start points, end points and angles of the slope in Theorem~\ref{almostmain} for certain graphs.}\label{graphtable}

\end{center}
\end{figure}

The degree sequence in Theorem~\ref{almostmain} is best possible in more than one sense for many graphs \(H\). For all graphs \(H\), one cannot allow significantly more than \(\omega n/h\) vertices to have degree below the `Koml\'os threshold', so in this sense the bound on the number of `small degree' vertices in Theorem~\ref{almostmain} is tight. Further, for many  graphs \(H\), we show that the degree sequence cannot start at a lower value and the angle of the `slope' in Figure~\ref{degseqkomthmfigure} is best possible. This is discussed in more depth in Section~\ref{Extremal}.

Theorem~\ref{almostmain} deals with almost perfect tilings. A natural question now is whether such a degree sequence strengthening also exists for tilings covering an \(x\)th proportion of vertices, as in Theorem~\ref{kom}. Indeed, the following result is a straightforward consequence of Theorem~\ref{almostmain}.

\begin{thm}\label{xdegseqKomlos}
Let \(x \in (0,1)  \) and \(H\) be a graph with \(\chi(H) = r\). Set \(\eta > 0\). Let \(\sigma:= \sigma (H)\),  \(h := |H|\) and \(\omega := \left(h - \sigma\right)/(r-1)\). Then there exists an \(n_0 = n_0(\eta,x, H) \in \mathbb{N}\) such that the following holds: Suppose \(G\) is a graph on \(n\geq n_0\) vertices with degree sequence \(d_1 \leq d_2 \leq \ldots \leq d_n\) such that
\[d_i \geq \left(g_H(x) - \frac{x\sigma}{h}\right)n + \frac{(r-1)x\sigma}{h - x\sigma}i \ \ \mbox{for all \ \(1 \leq i \leq \left(\frac{h - x\sigma}{(r-1)h}\right)n\).}\] 
Then \(G\) contains an \(H\)-tiling covering at least \((x-\eta)n\) vertices.
\end{thm}
\begin{figure}
\begin{center}
\begin{tikzpicture}

\draw (-4,5.5) -- (-4.2,5.5);
\draw[->] (-4,0) -- (-4,6.5);
\draw[->] (-4,0) -- (8,0);
\draw (-4,2.5) -- (-4.2,2.5);
\draw (0,0) -- (0,-0.2);
\draw (-4,5) -- (-4.2,5);
\node at (-4,-0.5) {\(d_1\)};
\draw (-4,0) -- (-4,-0.2);
\draw (-3.5,0) -- (-3.5,-0.2);
\node at (-3.5,-0.5) {\(d_2\)};
\draw[dash pattern=on 2pt off 4pt, thick] (-3,-0.5) -- (-0.5,-0.5);
\draw[dash pattern=on 2pt off 4pt, thick] (2.5,-0.5) -- (8,-0.5);
\node at (-5.5,2.5) {\(\left(1 - \frac{\omega+\sigma}{h}\right)n\)};
\node at (-5.25,5.5) {\(\left(1 - \frac{\omega}{h}\right)n\)};
\draw[black,thick] (-4,2.5) -- (0,5.5);
\draw[black,thick] (0,5.5) -- (8,5.5);
\node at (-4,6.9) {Degree};
\draw[dashed] (0,0) -- (0,5.45);
\draw[dashed] (1,0) -- (1,4.95);
\draw[dashed] (2,0) -- (2,4.45);
\draw (-4,3) -- (-4.2,3);
\draw (-4,3.5) -- (-4.2,3.5);
\draw (-4,4) -- (-4.2,4);
\draw (-4,4.5) -- (-4.2,4.5);
\draw (-4,5) -- (-4.2,5);
\node at (-6.2,3) {\(\left(g_H(2/3)- 2\sigma/3h\right)n\)};
\node at (-6.1,3.5) {\(\left(g_H(1/3)- \sigma/3h\right)n\)};
\node at (-5.85,4) {\(\left(1 - 1/(r-1)\right)n\)};
\node at (-5.25,4.5) {\(g_H(1/3)n\)};
\node at (-5.25,5) {\(g_H(2/3)n\)};
\draw[dash pattern=on 8pt off 4pt, thick] (-4,3) --  (1,5);
\draw[dash pattern=on 8pt off 4pt, thick] (1,5) -- (8,5);
\draw[dash pattern=on 5pt off 4pt, thick] (-4,3.5) --  (2,4.5);
\draw[dash pattern=on 5pt off 4pt, thick] (2,4.5) -- (8,4.5);
\draw[black,dashed] (-4,4) -- (8,4);
\draw (1,0) -- (1,-0.2);
\draw (2,0) -- (2,-0.2);
\node at (-0.15,-0.65) {\rotatebox{45}{\(d_{\frac{\omega n}{h}}\)}};
\node at (0.45,-1) {\rotatebox{45}{\(d_{\left(\frac{3h - 2\sigma}{3(r-1)h}\right)n}\)}};
\node at (1.45,-1) {\rotatebox{45}{\(d_{\left(\frac{3h - \sigma}{3(r-1)h}\right)n}\)}};
\end{tikzpicture}

\caption{The degree sequence in Theorem \ref{xdegseqKomlos} for \(x = 2/3\) (long dashed),  \(x = 1/3\) (medium dashed).}

\end{center}
\end{figure}
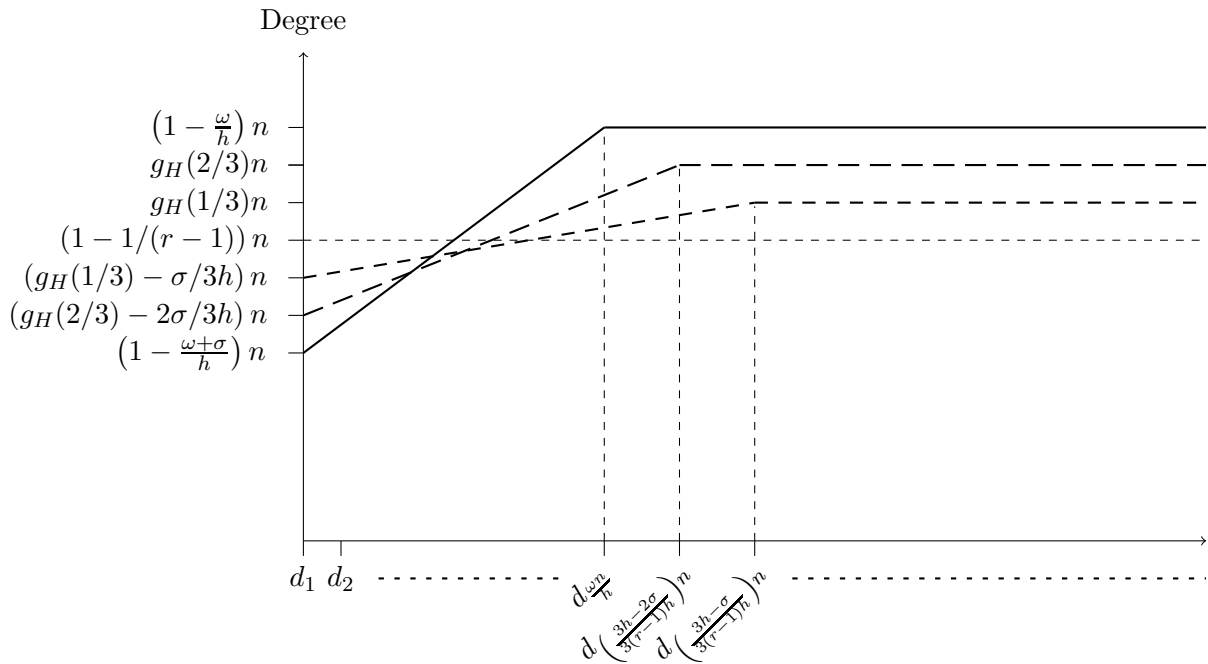

Theorem~\ref{xdegseqKomlos} is an improvement on Theorem~\ref{kom}. Indeed, Theorem~\ref{xdegseqKomlos} allows for almost \linebreak\((h - x\sigma)n/(r-1)h\) vertices to have degree below \(g_H(x)n\). Observe that as \(x\) approaches 0, the degree sequence condition in Theorem~\ref{xdegseqKomlos} tends towards the condition \(\delta(G) \geq (1-1/(r-1))n\), and thus  accords with the Erd\H{o}s--Stone theorem.


Piguet and Saumell~\cite[Theorem 1.3]{ps} recently proved another generalisation of Theorem~\ref{kom}. In their result they only require a certain fraction of the vertices to satisfy the degree condition of Theorem~\ref{kom}, and all other vertices have \emph{no} restriction on their degree (so some could even be isolated vertices). Note though that our result allows for more vertices to have small degree (i.e. smaller than the bound in Theorem~\ref{kom}), at a price of having some restriction of the degrees of these vertices. In the case of almost perfect \(H\)-tilings, Theorem~\ref{almostmain}
allows a large proportion of the vertices to have small degree, whilst in this case \cite[Theorem 1.3]{ps} corresponds precisely to Koml\'os' theorem (Theorem~\ref{komcor} above).

As well as considering minimum degree and degree sequence conditions, it is also natural to seek conditions on the density of a graph $G$ that forces an $H$-tiling covering a given fraction of the vertices of $G$. We remark though that only limited progress has been
made on this question (though Allen, B\"ottcher, Hladk\'y and Piguet~\cite{abhp} did resolve this problem in the case of $K_3$-tilings).

\smallskip

{\noindent \bf{Organisation of the paper.}}
The paper is organised as follows. In the next section we provide some essential notation and definitions. Then in Section~\ref{Extremal} we give  extremal examples for both Theorems~\ref{almostmain} and~\ref{xdegseqKomlos}. We introduce an `error term' version (Theorem~\ref{almostmainerror}) of Theorem~\ref{almostmain} in Section~\ref{precede} and show that it implies Theorem~\ref{almostmain}. Szemer\'{e}di's Regularity lemma and several auxiliary results are presented in Section~\ref{szereg}. Then in Section~\ref{Tools} we provide the tools that we will need to prove Theorem~\ref{almostmainerror}. We prove a result that iteratively constructs an almost perfect \(H\)-tiling and then use this result to prove Theorem~\ref{almostmainerror} in Section~\ref{almost}. To conclude Section~\ref{almost}, we show that Theorem~\ref{almostmain} implies Theorem~\ref{xdegseqKomlos}.

\section{Notation and Definitions}

Let \(G\) be a graph. We define \(V(G)\) to be the vertex set of \(G\) and \(E(G)\) to be the edge set of \(G\). Let \(X \subseteq V(G)\). Then \(G[X]\) is the \textit{graph induced by \(X\) on \(G\)} and has vertex set \(X\) and edge set \(E(G[X]) := \{xy \in E(G): x,y \in X\}\). We also define \(G\setminus X\) to be the graph with vertex set \(V(G)\setminus X\) and edge set \(E(G\setminus X):= \{xy \in E(G): x,y \in V(G)\setminus X\}\). For each \(x \in V(G)\), we define the \textit{neighbourhood of \(x\) in \(G\)} to be \(N_G(x):= \{y\in V(G): xy \in E(G)\}\) and define \(d_G(x) := |N_G(x)|\). We drop the subscript \(G\) if it is clear from context which graph we are considering. 
We write $d_G(x,X)$ for the number of edges in $G$ that $x$ sends to vertices in $X$. Given a subgraph $G'\subseteq G$, we will write $d_G(x,G'):=d_G(x,V(G'))$.
Let \(A, B \subseteq V(G)\) be disjoint. Then we define \(e_G(A,B):=|\{xy \in E(G): x\in A, y\in B\}|\).

Let \(t \in \mathbb{N}\). We define the \emph{blow-up} \(G(t)\) to be the graph constructed by first replacing each vertex \(x \in V(G)\) by a set \(V_x\) of \(t\) vertices and then replacing each edge \(xy \in E(G)\) with the edges of the complete bipartite graph with vertex sets \(V_x\) and \(V_y\).

Let \(v \in \mathbb{N}\). We will refer to a vertex class of size \(v\) of \(G\) as a \textit{\(v\)-class} of \(G\). Set \(r, \sigma, \omega \in \mathbb{N}\) and \(\sigma <\omega\). We define the \textit{\(r\)-partite bottle graph \(B\) with neck \(\sigma\) and width \(\omega\)} to be the complete \(r\)-partite graph with one \(\sigma\)-class and \((r-1)\) \(\omega\)-classes.

Let \(i \in \mathbb{N}\) and \(H_1, H_2, \ldots, H_i\) be a collection of graphs. We define an \textit{\((H_1, H_2, \ldots, H_i)\)-tiling in \(G\)} to be a collection of vertex-disjoint copies of graphs from the set \(\{H_1, H_2, \ldots, H_i\}\) in \(G\). An \((H_1, H_2, \ldots, H_i)\)-tiling is called \textit{perfect} if it covers all vertices in \(G\). 

We write \(0 < a \ll b \ll c < 1\) to mean that we can choose the constants \(a,b,c\) from right to left. More precisely, there exist non-decreasing functions \(f: (0,1] \to (0,1]\) and \(g: (0,1] \to (0,1]\) such that for all \(a \leq f(b)\) and \(b \leq g(c)\) our calculations and arguments in our proofs are correct. Larger hierarchies are defined similarly. Note that \(a \ll b\) implies that we may assume e.g. \(a < b\) or \(a < b^2\).

\section{Extremal Examples}\label{Extremal}
In this section we present three extremal examples. The first demonstrates that the `slope' of the degree sequence in Theorem~\ref{almostmain} is best possible for bottle graphs. 
The second shows that for many graphs $H$, the degree sequence in Theorem~\ref{almostmain} `starts' at the correct place.
The third shows that, for any graph \(H\), to ensure an \(H\)-tiling covering at least \((x - \eta)n\) vertices we cannot have significantly more than \((h-x\sigma)n/(r-1)h\) vertices with degree below the `Koml\'{o}s threshold' of \(g_H(x)n\). 

\smallskip

\noindent
{\bf Extremal Example 1.} \textnormal{Set \(\eta \in \mathbb{R}\). Let \(B\) be an \(r\)-partite bottle graph with neck \(\sigma\) and width \(\omega\), where \(b := |B|\). The following extremal example \(G\) on \(n\) vertices demonstrates that Theorem~\ref{almostmain} is best possible for such graphs \(B\), in the sense that \(G\) satisfies the degree sequence of Theorem~\ref{almostmain} except for a small linear part that only just fails the degree sequence, but does not contain a \(B\)-tiling covering all but at most \(\eta n\) vertices.} 

\begin{prop}\label{exex1}
Set \(\eta \in \mathbb{R}\) and \(n \in \mathbb{N}\) such that \(0 < 1/n \ll \eta \ll 1\). Let \(B\) be a bottle graph with neck \(\sigma\) and width \(\omega\), where \(b := |B|\). 
Additionally assume that $b$ divides $n$.
Then for any \(1 \leq k < \omega n/b - 2\eta n\), there exists a graph \(G\) on \(n\) vertices whose degree sequence \(d_1 \leq \ldots \leq d_n\) satisfies 
\[d_i \geq \left(1 - \frac{\omega+\sigma}{b}\right)n + \frac{\sigma}{\omega}i \ \ \mbox{for all \ \(i \in \{1, \ldots, k-1, k+2\eta n +1, \ldots, \omega n/b\)\},}\] 
\[d_i = \left(1 - \frac{\omega+\sigma}{b}\right)n + \left\lceil\frac{\sigma}{\omega}k\right\rceil \ \ \mbox{for all \ \(k \leq i \leq k+2\eta n\),}\]  but such that \(G\) does not contain a \(B\)-tiling covering all but at most \(\eta n\) vertices.
\end{prop}

\begin{proof}\textnormal{Let \(G\) be the graph on \(n\) vertices with \(r\) vertex classes \(V_1, \ldots, V_{r}\) where \(|V_1| = \sigma n/b\) and \(|V_2| = |V_3| = \ldots = |V_r| = \omega n/b\). Label the vertices of \(V_1\) as \(a_1, a_2, \ldots, a_{\sigma n/b}\). Similarly, label the vertices of \(V_2\) as \(c_1, c_2, \ldots, c_{\omega n/b}\). The edge set of \(G\) is constructed as follows.}

\textnormal{Firstly, let \(G\) have the following edges:}

\begin{itemize}
     \item \textnormal{All edges with an endpoint in \(V_1\) and the other endpoint in \(V(G)\setminus V_2\), in particular \(G[V_1]\) is complete;}
     \item \textnormal{All edges with an endpoint in \(V_i\) and the other endpoint in \(V(G)\setminus (V_1 \cup V_i) \) for \(2  \leq i \leq r\);}
     \item \textnormal{Given any \(1 \leq i \leq \omega n/b\) and \(j \leq \lceil\sigma i/\omega\rceil\) include all edges  \(c_i a_j\).}
\end{itemize}

\textnormal{
So at the moment $G$ does satisfy the degree sequence in Theorem~\ref{almostmain}; we therefore modify $G$ slightly.
For all \(k~+~1~\leq~i~\leq~k~+~2\eta n\) and \(\lceil\sigma k/\omega \rceil + 1\leq j \leq \lceil\sigma (k+2\eta n)/\omega\rceil\) delete each edge between \(c_i\) and \(a_j\).
One can easily check that \(G\) satisfies the degree sequence in the statement of the proposition. In particular, the vertices of degree \(\left(1 - \frac{\omega+\sigma}{b}\right)n + \lceil\frac{\sigma}{\omega}k\rceil  \) are \(c_{k},\dots,
c_{k+2\eta n}\).}

Define \(A := \{a_1, \ldots, a_{\lceil\sigma k/\omega\rceil}\}\) and \(C := \{c_1, \ldots, c_{k + 2\eta n}\}\). Note that there are no edges between $C$ and $V_1 \setminus A$ in $G$.

\begin{claim} \label{exexampclaim} Let \(T\) be a \(B\)-tiling of \(G\). Then \(T\) does not cover at least \(3\eta n/2\) vertices in \(C\).
\end{claim}
\textnormal{Consider any copy \(B'\) of \(B\) in \(G\) that contains an element of \(C\). 
As $C$ is an independent set in $G$,  \(B'\) contains at most \(\omega\) elements from \(C\).
Since
there are no edges between $C$ and $V_1 \setminus A$ in $G$,
 \(B'\) contains at least \(\sigma\) vertices in \(A\). 
This implies that at most \(\lceil\sigma k/\omega\rceil(\omega/\sigma) < k + \eta n/2\) vertices in \(C\) can be covered by \(T\). Since \(|C| = k + 2 \eta n\), we have that \(T\) does not cover at least \(3\eta n/2\) vertices in \(C\). Therefore, Claim \ref{exexampclaim} holds. Hence \(G\) does not have a \(B\)-tiling covering all but at most \(\eta n\) vertices.}\end{proof}

Proposition~\ref{exex1} implies that for bottle graphs $B$, the degree sequence in Theorem~\ref{almostmain} cannot be lowered significantly in a small part of the degree sequence and still ensure an almost perfect $B$-tiling; so the `slope' of the degree sequence
in Theorem~\ref{almostmain} cannot be improved upon. It would be interesting to find other classes of graphs $H$ for which the slope in Theorem~\ref{almostmain} is also best possible; we suspect though that there are  graphs $H$ where the slope is not best possible.

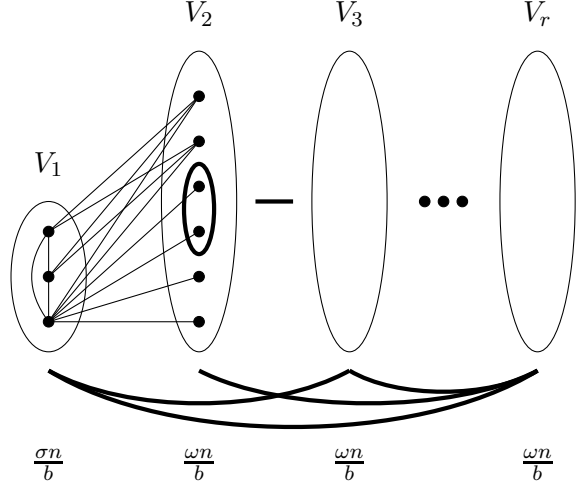
\begin{figure}
\begin{center}
\begin{tikzpicture}
\draw (-2,1) ellipse (0.5 and 1);
\draw (0,2) ellipse (0.5 and 2);
\draw (2,2) ellipse (0.5 and 2);
\filldraw[black] (3,2) circle (2pt);
\filldraw[black] (3.25,2) circle (2pt);
\filldraw[black] (3.5,2) circle (2pt);
\draw (4.5,2) ellipse (0.5 and 2);
\filldraw[black] (-2,0.4) circle (2pt);
\filldraw[black] (-2,1) circle (2pt);
\filldraw[black] (-2,1.6) circle (2pt);
\filldraw[black] (0,0.4) circle (2pt);
\filldraw[black] (0,1) circle (2pt);
\filldraw[black] (0,1.6) circle (2pt);
\filldraw[black] (0,2.2) circle (2pt);
\filldraw[black] (0,2.8) circle (2pt);
\filldraw[black] (0,3.4) circle (2pt);
\draw[black, ultra thick] (0,1.9) ellipse (0.2 and 0.6);
\draw (-2,0.4) -- (0,0.4);
\draw (-2,0.4) -- (0,1);
\draw (-2,0.4) -- (0,1.6);
\draw (-2,0.4) -- (0,2.2);
\draw (-2,0.4) -- (0,2.8);
\draw (-2,0.4) -- (0,3.4);
\draw (-2,1) -- (0,2.8);
\draw (-2,1) -- (0,3.4);
\draw (-2,1.6) -- (0,2.8);
\draw (-2,1.6) -- (0,3.4);
\node at (-2,-1.5) {\(\frac{\sigma n}{b}\)};
\node at (0,-1.5) {\(\frac{\omega n}{b}\)};
\node at (2,-1.5) {\(\frac{\omega n}{b}\)};
\node at (4.5,-1.5) {\(\frac{\omega n}{b}\)};
\node at (-2,2.5) {\(V_1\)};
\node at (0,4.5) {\(V_2\)};
\node at (2,4.5) {\(V_3\)};
\node at (4.5,4.5) {\(V_r\)};
\draw[ultra thick] (0.75,2) -- (1.25,2);
\draw[ultra thick] (-2,-0.25) .. controls (-0.95,-0.83) and (0.95,-0.83) .. (2,-0.25);
\draw[ultra thick] (0,-0.25) .. controls (1.125,-0.83) and (3.375,-0.83) .. (4.5,-0.25);
\draw[ultra thick] (2,-0.25) .. controls (2.625,-0.625) and (3.875,-0.625) .. (4.5,-0.25);
\draw[ultra thick] (-2,-0.25) .. controls (-0.375,-1.25) and (2.875,-1.25) .. (4.5,-0.25);
\draw (-2,0.4) -- (-2,1);
\draw (-2,1) -- (-2,1.6);
\draw (-2,0.4) .. controls (-2.3,0.8) and (-2.3,1.2) .. (-2,1.6);
\end{tikzpicture}

\caption{An example of a graph \(G\) in Proposition~\ref{exex1} where \(\sigma = 1, \omega = 2\).}
\end{center}
\end{figure}

\smallskip

\noindent
{\bf Extremal Example 2.}
The next example shows that for many graphs $H$, Theorem~\ref{almostmain} is best possible in the sense that we cannot start the degree sequence at a significantly lower value.

\begin{prop}
Let $H$ be an $r$-partite graph so that, for every $x \in V(H)$, $H[N (x)]$ is $(r-1)$-partite. Let $h:=|H|$, \(\sigma:=\sigma(H)\)  and set \(\omega := (h-\sigma)/(r-1)\). 
Additionally suppose $\sigma <\omega$.
Let $0<1/n \ll \eta \ll (\omega -\sigma )/h$ where $h$ divides $n$. Then there is an $n$-vertex graph $G$ with
\begin{itemize}
\item[(i)] $\lfloor \eta n \rfloor +1$ vertices of degree $(1-\frac{\omega +\sigma}{h})n$,
\item[(ii)] all other vertices have degree at least $(1-1/\chi _{cr} (H))n=(1-\omega/h)n$,
\end{itemize}
and  $G$ does not have an $H$-tiling covering all but at most $\eta n$ vertices.
\end{prop}
\begin{proof}
Let \(G\) be the complete $r$-partite graph on \(n\) vertices with  vertex classes \(V_1, \ldots, V_{r}\) where \(|V_1| = \sigma n/h+\lfloor \eta n \rfloor +1\), 
$|V_2|=\omega n/h -\lfloor \eta n \rfloor -1$ and
\( |V_3| = \ldots = |V_r| = \omega n/h\).
Let $V'\subseteq V_1$ be of size $\lfloor \eta n \rfloor +1$. Delete from $G$ all edges with one  endpoint in $V'$ and the other in $V_2$.
By construction $G$ satisfies (i) and (ii). Note that since the neighbourhood of any $x \in V'$ induces an $(r-2)$-partite subgraph of $G$, no vertex in $V'$ lies in a copy of $H$ in $G$.
So $G$ does not have an $H$-tiling covering all but at most $\eta n$ vertices.
\end{proof}

\noindent
{\bf Extremal Example 3.}
 \textnormal{Set \(\eta \in \mathbb{R}\) and \(x \in (0,1]\). Let \(H\) be a graph with \(\chi(H) =: r\). Let $h:=|H|$, \(\sigma := \sigma(H)\) and set \(\omega := (h-\sigma)/(r-1)\).
Define \(g_H(1) := 1 - \omega/h\).
 We give an extremal example \(G\) on \(n\) vertices which satisfies the degree sequence of Theorem~\ref{xdegseqKomlos} except that \((h-x\sigma)n/(r-1)h + \eta n\) vertices have degree at most \((g_H(x) - \eta)n\), but does not contain an \(H\)-tiling covering at least \((x -\eta)n\) vertices. }

\begin{prop}
Set \(\eta \in \mathbb{R}\) and \(x \in (0,1]\). Let \(H\) be a graph with \(\chi(H) =: r\). Let $h:=|H|$, \(\sigma := \sigma(H)\) and set \(\omega := (h-\sigma)/(r-1)\). Then there exists a graph \(G\) on \(n\) vertices whose degree sequence \(d_1 \leq \ldots \leq d_n\) satisfies 
$$d_i = \left(g_H(x)-\eta\right)n \ \ \text{for all } \ i \leq \frac{h-x\sigma}{(r-1)h}n + \eta n,$$
$$d_i \geq g_H(x) n \ \ \text{ for all  } \ i > \frac{h-x\sigma}{(r-1)h}n + \eta n,$$
but such that \(G\) does not contain an \(H\)-tiling covering at least \((x-\eta)n\) vertices.
\end{prop}
\begin{proof} \textnormal{Let \(G\) be the complete $r$-partite graph on \(n\) vertices with vertex classes \(V_1, \ldots, V_r\) such that} 
\begin{itemize}
\item \(|V_1| = \frac{x\sigma n}{h} - \eta n\), 
\item \(|V_2| = \frac{(h-x\sigma)n}{(r-1)h} + \eta n\), 
\item \(|V_3| = \ldots = |V_r| = \frac{(h-x\sigma)n}{(r-1)h}\).
\end{itemize} 

\textnormal{Consider any \(H\)-tiling \(T\) of \(G\). Observe that \(T\) can contain at most \(xn/h - \eta n/\sigma\) copies of \(H\). Indeed, 
to attain this bound one requires that
 all colour classes of size \(\sigma\) in copies of \(H\) are placed into \(V_1\). Hence at most \(x(r~-~1)\omega n/h~-~(r-1)\omega\eta n/\sigma\) vertices are covered by \(T\) in \(V_2 \cup \ldots \cup V_r\). Thus at most \((x - \eta)n - (r-1)\omega\eta n/\sigma\) vertices are covered by \(T\). Hence \(G\) does not contain an \(H\)-tiling covering at least \((x - \eta)n\) vertices}
\end{proof}

\section{Deriving Theorem~\ref{almostmain} from a weaker result}\label{precede}

To prove Theorem~\ref{almostmain} we will first prove the following `error term' version.

\begin{thm} \label{almostmainerror}
Let \(\eta > 0\) and \(H\) be a graph with \(\chi(H) = r\). Let $h:=|H|$, \(\sigma:=\sigma(H)\) and set \(\omega := (h-\sigma)/(r-1)\). Then there exists an \(n_0 = n_0(\eta, H) \in \mathbb{N}\) such that the following holds: Suppose \(G\) is a graph on \(n\geq n_0\) vertices with degree sequence \(d_1 \leq d_2 \leq \ldots \leq d_n\) such that
\[d_i \geq \left(1 - \frac{\omega+\sigma}{h}\right)n + \frac{\sigma}{\omega}i + \eta n \ \ \mbox{for all \ \(1 \leq i \leq \frac{\omega n}{h}\).}\]
Then \(G\) has an \(H\)-tiling covering all but at most \(\eta n\) vertices.
\end{thm}

Theorem~\ref{almostmainerror} implies Theorem~\ref{almostmain}. Indeed, a simple argument (as in~\cite{Komlos}) allows us to remove the error terms.

\begin{proofff}
Set \(0< \tau \ll \eta, 1/h\) and let $n\in \mathbb N$ be sufficiently large. 
Suppose $G$ is an $n$-vertex graph as in the statement of Theorem~\ref{almostmain}.
Let \(A\) be a set of \(\tau n\) vertices and define \(G^{*}\) to be the graph with vertex set \(V(G) \cup A\) and edge set \(E(G^{*}) := E(G) \cup \{xy : x \in V(G) \cup A, y \in A, x \neq y\}\). Then \(G^{*}\) has degree sequence \(d_{G^{*}, 1} \leq d_{G^{*}, 2} \leq \ldots \leq d_{G^{*}, (1 + \tau)n}\) where 
\begin{align*} d_{G^{*},i} & \geq  \left(1 - \frac{\omega+\sigma}{h}\right)n + \frac{\sigma}{\omega}i + \tau n 
 \geq  \left(1 - \frac{\omega+\sigma}{h}\right)(1+\tau)n + \left(\frac{\sigma}{\omega}i+ \frac{\sigma\tau}{h}n\right) + \frac{\omega\tau}{h}n 
\end{align*} 
for all  \(1 \leq i \leq \frac{\omega n}{h}\) and
\begin{align*} 
d_{G^{*},i} \geq \left(1 - \frac{\omega+\sigma}{h}\right)(1+\tau)n + \frac{\sigma}{\omega}i + \frac{\omega \tau}{2h}(1 + \tau)n 
\end{align*}
for all $\frac{\omega n}{h} \leq   i \leq \frac{\omega (1+ \tau)n}{h}$.
By Theorem~\ref{almostmainerror} we have that \(G^{*}\) has an \(H\)-tiling \(T\) covering all but at most \(\frac{\omega \tau}{2h}(1+\tau)n\) vertices.

Now, remove every copy of \(H\) from \(T\) that contains a vertex in \(A\). Then we have removed at most \((h-1)\tau n\) vertices from \(V(G) \subset V(G^{*})\). Moreover, this implies that there exists an \(H\)-tiling in \(G\) covering all but at most \((h-1)\tau n + \frac{\omega \tau}{2h}(1+\tau) n\) vertices. Since \((h-1)\tau n + \frac{\omega \tau}{2h}(1+\tau) n < \eta n\), Theorem~\ref{almostmain} holds.\qed
\end{proofff}

\medskip

{\noindent \bf Outline of the proof of Theorem~\ref{almostmainerror}.}
The aim of the rest of the paper is to prove Theorem~\ref{almostmainerror}; we now outline the proof of this result. We first show that it suffices to prove Theorem~\ref{almostmainerror} in the case when $H=B$, a bottle graph with neck
$\sigma$ and width $\omega$ (where $\sigma <\omega$). In particular, Theorem~\ref{almostmainerror} is already known in the case when $H$ is a balanced $r$-partite graph~\cite{Tregs}.

We then employ a variant of an idea of Koml\'os~\cite{Komlos}. Roughly speaking the idea is as follows:
Let $B^*$ be a suitably large blown-up copy of $B$.
We apply the Regularity lemma (Lemma~\ref{degformreglemma}) to obtain a reduced graph $R$ of $G$. If $R$ contains an
almost perfect $B^*$-tiling then one can rather straightforwardly conclude 
 that $G$ contains an almost perfect $B$-tiling, as required (for this we apply Lemma~\ref{almosttilingRtoG}).
 Otherwise, suppose that the largest $B^*$-tiling  in $R$
covers precisely $d\leq (1-o(1))|R|$ vertices. We then show that, for some $t \in \mathbb N$, there is a  $B^*$-tiling in the blow-up $R(t)$ of $R$ covering substantially more than $dt$ vertices.
 Thus, crucially, the largest $B^*$-tiling in $R(t)$ covers a higher proportion of vertices than the largest $B^*$-tiling in $R$. By repeating this argument, we obtain a blow-up $R'$ of $R$ that contains an almost perfect $B^*$-tiling. We then
show that this implies  $G$ contains an almost perfect $B$-tiling, as desired. 

Other applications of this general method have been used in the past~\cite{cz, hlad, Tregs}. Note however, our approach  has different challenges.
Indeed, the process of moving from a $B^*$-tiling $\mathcal B$ in $R$ to a proportionally larger $B^*$-tiling in $R(t)$
is rather subtle. 
In particular, what we would like to do is conclude that one can find a tiling $\mathcal B_0$ (not necessarily of copies of $B^*$) in $R$ that covers a larger proportion of the vertices in $R$ \emph{and} when one takes a suitable blow-up $R(t)$ of $R$,
then $\mathcal B_0$ corresponds to a $B^*$-tiling in $R(t)$. However,
the vertices in $R$ that are uncovered by $\mathcal B$ could perhaps all be `small degree' vertices (i.e. they do not have degree as large as that in Theorem~\ref{komcor}). This is a barrier to finding such a special tiling $\mathcal B_0$.
(Intuitively, one can imagine that if one has large degree vertices outside of $\mathcal B$ then one can glue such vertices onto $\mathcal B$ in such a way to obtain our desired tiling $\mathcal B_0$.)
In this case, one has to (through perhaps many steps) modify $\mathcal B$ and then blow-up $R$ to obtain an intermediate blow-up $R(t')$ of $R$ such that (i)  there is a $B^*$-tiling $\mathcal B'$ in $R(t')$ that covers the same proportion of
vertices compared to the tiling  $\mathcal B$ in $R$ and (ii) many of the vertices in $R(t')$ uncovered by $\mathcal B'$ are now such that they can be `glued' onto $\mathcal B'$ to obtain our desired larger tiling $\mathcal B_0$.

Despite these technicalities the proof of Theorem~\ref{almostmainerror} is perhaps surprisingly short.
The main work of the proof is encoded in  Lemma~\ref{expandorswaplemma}, which ensures one can modify the tiling $\mathcal B$ as above.

\section{Szemer\'{e}di's Regularity lemma and auxiliary results}\label{szereg}

A key tool in the proof of Theorem~\ref{almostmainerror} is Szemer\'{e}di's Regularity lemma \cite{sze}. To state this lemma we will need the following notion of \(\varepsilon\)-regularity.

\begin{define}
\textnormal{Let \(G =(A,B)\) be a bipartite graph with vertex classes \(A\) and \(B\). We define the \textit{density} of \(G\) to be \[d_G(A,B) := \frac{e_G(A,B)}{|A||B|}.\] Set \(\varepsilon > 0\). We say that \(G\) is \textit{\(\varepsilon\)-regular} if for all \(X \subseteq A\) and \(Y \subseteq B\) with \(|X| > \varepsilon|A|\) and \(|Y| > \varepsilon|B|\) we have that \(|d_G(X,Y) - d_G(A,B)| < \varepsilon\).} 
\end{define}

\begin{lem} \textnormal{(Degree form of Szemer\'{e}di's Regularity lemma).} \label{degformreglemma} 
Let \(\varepsilon \in (0,1)\) and \(M' \in \mathbb{N}\). Then there exist natural numbers \(M\) and \(n_0\) such that for any graph \(G\) on \(n \geq n_0\) vertices and any \(d \in (0,1)\) there is a partition of the vertices of \(G\) into subsets \(V_0, V_1, \ldots, V_k\) and a spanning subgraph \(G'\) of \(G\) such that the following hold:

\begin{itemize}
    \item \(M' \leq k \leq M\);
    \item \(|V_0| \leq \varepsilon n\);
    \item \(|V_1| = \ldots = |V_k| =: q\);
    \item \(d_{G'}(x) > d_{G}(x) - (d+ \varepsilon)n\) for all \(x \in V(G)\);
    \item \(e(G'[V_i]) = 0\) for all \(i \geq 1\);
    \item For all \(1 \leq i,j \leq k\) with \(i \neq j\) the pair \((V_i, V_j)_{G'}\) is \(\varepsilon\)-regular and has density either \(0\) or at least \(d\).
\end{itemize}
\end{lem}

We call \(V_1, \ldots, V_{k}\) the \textit{clusters} of our partition, \(V_0\) the exceptional set and \(G'\) the \textit{pure graph}. We define the \textit{reduced graph} \(R\) of \(G\) with parameters \(\varepsilon\), \(d\) and \(M'\) to be the graph whose vertex set is \(V_1, \ldots, V_k\) and in which \(V_iV_j\) is an edge if and only if \((V_i, V_j)_{G'}\) is \(\varepsilon\)-regular with density at least \(d\). Note also that \(|R| = k\).

The proof of the next result is analogous to that of Lemma 5.2 in \cite{Tregs}. It states that the degree sequence of \(G\) in Theorem~\ref{almostmainerror} is `inherited' by its reduced graph \(R\). 

\begin{lem} \label{inheriteddegseqreduced}
Set \(M', n_0 \in \mathbb{N}\) and \(\varepsilon, d, \eta, b, \omega, \sigma\) to be positive constants such that \(1/n_0 \ll 1/M' \ll \varepsilon \ll d \ll \eta \ll 1/b\) and where \(\omega + \sigma \leq b\). Suppose \(G\) is a graph on \(n \geq n_0\) vertices with degree sequence \(d_1 \leq d_2 \leq \ldots \leq d_n\) such that
\begin{equation} 
\label{originaldegseqreduced} d_i \geq \frac{b-\omega-\sigma}{b}n + \frac{\sigma}{\omega}i + \eta n \ \ \mbox{for all \(1 \leq i \leq \frac{\omega n}{b}\).}
\end{equation}
Let \(R\) be the reduced graph of \(G\) with parameters \(\varepsilon\), \(d\) and \(M'\) and set \(k:= |R|\). Then \(R\) has degree sequence \(d_{R,1} \leq d_{R,2} \leq \ldots \leq d_{R,k}\) such that
\begin{equation} \label{inheriteddegseqreducedeq} 
d_{R,i} \geq \frac{b-\omega-\sigma}{b}k + \frac{\sigma}{\omega}i + \frac{\eta k}{2} \ \ \mbox{for all \(1 \leq i \leq \frac{\omega k}{b}\).}
\end{equation}
\end{lem}

\begin{proof}
Let \(V_1, \ldots, V_k\) be the clusters of \(G\) and \(V_0\) the exceptional set, and let \(G'\) be the pure graph of \(G\). Set \(q := |V_1| = \ldots = |V_k|\). Clearly we may assume \(d_{R}(V_1) \leq d_{R}(V_2) \leq \ldots \leq d_{R}(V_k)\). Now consider any \(i \leq \frac{\omega k}{b}\). Set \(S := \cup_{1\leq j \leq i} V_j\). Then \(|S| = qi \leq \frac{\omega qk}{b} \leq \frac{\omega n}{b}\). Thus by (\ref{originaldegseqreduced}) there exists a vertex \(x \in S\) such that 
\(d_{G}(x) \geq d_{qi} \geq \frac{b-\omega-\sigma}{b}n + \left(\frac{\sigma}{\omega}\right)qi + \eta n\). Suppose that \(x \in V_j\) where \(1 \leq j \leq i\). Since we have that \(kq \leq n\), Lemma~\ref{degformreglemma} implies that
\begin{align*} 
d_{R}(V_j) \geq \frac{d_{G'}(x) - |V_0|}{q} &\geq \frac{1}{q}\left(\frac{b-\omega-\sigma}{b}n + \left(\frac{\sigma}{\omega}\right)qi + \eta n - (d + 2\varepsilon)n)\right) \\ &\geq \frac{b-\omega-\sigma}{b}k + \frac{\sigma}{\omega}i + \frac{\eta k}{2}. 
\end{align*}
 Since \(d_{R,i} = d_{R}(V_i) \geq d_{R}(V_j)\) we have that (\ref{inheriteddegseqreducedeq}) holds. 
\end{proof}
We will also apply the following well-known fact.

\begin{fact} \label{slicinglemma}
Let \(0 < \varepsilon < \alpha\) and \(\varepsilon':= \max\{\varepsilon/\alpha, 2\varepsilon\}\). Let \((A,B)\) be an \(\varepsilon\)-regular pair of density \(d\). Suppose \(A' \subseteq A\) and \(B' \subseteq B\) where \(|A'| \geq \alpha |A|\) and \(|B'| \geq \alpha |B|\). Then \((A', B')\) is an \(\varepsilon'\)-regular pair with density \(d'\) where \(|d'-d| < \varepsilon\).
\end{fact}

\begin{lem} \textnormal{(Key Lemma \cite{KomlosSimonovits})} \label{keylem}
Suppose that \(0 < \varepsilon < d\), that \(q, t \in \mathbb{N}\) and that \(R\) is a graph where \(V(R) = \{v_1, \ldots, v_k\}\). We construct a graph \(G\) as follows: Replace every vertex \(v_i \in V(R)\) by a set \(V_i\) of \(q\) vertices and replace each edge of \(R\) by an \(\varepsilon\)-regular pair of density at least \(d\). For each \(v_i \in V(R)\), let \(U_i\) denote the set of \(t\) vertices in \(R(t)\) corresponding to \(v_i\). Let \(H\) be a subgraph of \(R(t)\) with maximum degree \(\Delta\), and set \(h := |H|\). Set \(\delta := d - \varepsilon\) and \(\varepsilon_0 := \delta^{\Delta}/(2 + \Delta)\). If \(\varepsilon \leq \varepsilon_0\) and \(t-1 \leq \varepsilon_0q\) then there are at least $(\varepsilon_0 q)^h $ labelled copies of \(H\) in \(G\) so that if \(x \in V(H)\) lies in \(U_i\), then \(x\) is embedded into \(V_i\) in \(G\).
\end{lem}

\section{Tools for proving Theorem~\ref{almostmainerror}}\label{Tools}

In this section we provide further tools that we will need to prove Theorem~\ref{almostmainerror}. The following lemma is a special case of Lemma 11 in \cite{Komlos} (which in turn is easily implied by the Key lemma above).

\begin{lem} \label{almosttilingRtoG}
Set \(0 < \beta < 1/2\) and let \(B\) be the bottle graph with neck \(\sigma\) and width \(\omega\). Set  \(d \in (0,1)\). Then there exists an \(\varepsilon' > 0\) such that for all \(\varepsilon \leq \varepsilon'\) the following holds for all 
$q \in \mathbb N$: Let \(G\) be a graph constructed from \(B\) by replacing every vertex of \(B\) by \(q\) vertices and replacing the edges of \(B\) with \(\varepsilon\)-regular pairs of density at least \(d\). Then \(G\) has a \(B\)-tiling covering all but at most a \(\beta\)-proportion of the vertices in \(G\).
\end{lem}
Given a bottle graph $B$, the next lemma ensures various blown-up copies of graphs contain perfect $B$-tilings.

\begin{lem} \label{mtblowuplem}
Set \(m \in \mathbb{N}\). Let \(B\) be an \(r\)-partite bottle graph with neck \(\sigma\) and width \(\omega\), where \(b := |B|\) and \(\sigma < \omega\). Define \(B'\) to be the \(r\)-partite bottle graph with neck \(\sigma\) and width \(\omega - 1\) and let \(B^{*} := B(m)\). Define \(t := (\omega - \sigma)b\). Then \(B(mt)\), \(B^{*}(mt)\), \(B'(mt)\) and \(K_r(mt)\) all have perfect \(B^{*}\)-tilings.
\end{lem}

\begin{proof}
Clearly \(B(mt)\) and \(B^{*}(mt)\) both have perfect \(B^{*}\)-tilings. It remains to show that \(B'(mt)\) and \(K_r(mt)\) have perfect \(B^{*}\)-tilings. 

For \(K_r(mt)\), tile \((\omega - \sigma)r\) copies of \(B^{*}\) into \(K_r(mt)\) such that their \((\sigma m)\)-classes are distributed evenly amongst the \(r\) vertex classes of \(K_r(mt)\). Indeed, we can view this as tiling \((\omega - \sigma)\) collections of \(r\) copies of \(B^{*}\) into \(K_r(mt)\) such that, for each collection \(C\), each vertex class of \(K_r(mt)\) contains the (\(\sigma m\))-class of precisely one copy of \(B^{*}\) in \(C\). 

For \(B'(mt)\), firstly tile \((\omega - 1 - \sigma)b\) vertex-disjoint copies of \(B^{*}\) into \(B'(mt)\) such that each (\(\sigma m\))-class is placed into the (\(\sigma mt\))-class in \(B'(mt)\). So our current \(B^{*}\)-tiling covers all but \(\sigma mt - \sigma m(\omega - 1 - \sigma)b = \sigma mb\) vertices in the (\(\sigma mt\))-class in \(B'(mt)\) and all but \((\omega - 1)mt - \omega m(\omega - 1 - \sigma)b = \sigma mb\) vertices in each (\(\omega mt\))-class in \(B'(mt)\). Then the remaining vertices to be covered in \(B'(mt)\) form a \(K_r(\sigma mb)\) which can be tiled with \(\sigma r\) copies of \(B^{*}\).
\end{proof}
The next result states that the degree sequence of \(G\) in Theorem~\ref{almostmainerror} is inherited by any blown-up copy of \(G\).

\begin{prop}\label{inheriteddegseqG(s)} 
Let \(n, s \in \mathbb{N}\) and \(b, \omega, \sigma > 0\) such that \(\omega n > b\) and \(\omega + \sigma \leq b\). Set \(\eta > 0\). Suppose \(G\) is a graph on \(n\) vertices with degree sequence \(d_{G,1} \leq d_{G,2} \leq \ldots \leq d_{G,n}\) such that 
\[d_{G,i} \geq \frac{b-\omega-\sigma}{b}n + \frac{\sigma}{\omega}i + \eta n \ \ \mbox{for all \(1 \leq i \leq \frac{\omega n}{b}\).}\]
 Then \(\bar{G} := G(s)\) has degree sequence \(d_{\bar{G}, 1} \leq d_{\bar{G}, 2} \leq \ldots \leq d_{\bar{G}, ns}\) such that
\[d_{\bar{G}, i} \geq \frac{b-\omega-\sigma}{b}ns + \frac{\sigma}{\omega}i + \left(\eta n - \frac{\sigma}{\omega}\right)s \ \ \mbox{for all \(1 \leq i \leq \frac{\omega ns}{b}\).}\]
\end{prop}

\begin{proof}
For any \(1 \leq j \leq ns\) we see that 
$$d_{\bar{G},j} = s \cdot d_{G,\lceil j/s\rceil}.$$ 
Suppose that \(j \leq \frac{\omega ns}{b} - s\). Then \(\lceil j/s\rceil \leq \frac{\omega n}{b}\) and we have 
$$d_{\bar{G},j} \geq \frac{b-\omega-\sigma}{b}ns + \frac{\sigma}{\omega}\lceil j/s\rceil s  + \eta ns \geq \frac{b-\omega-\sigma}{b}ns + \frac{\sigma}{\omega}j  + \eta ns.$$
 In particular, if we take any \(i \leq \frac{\omega ns}{b}\) we have 
$$d_{\bar{G}, i} \geq \frac{b-\omega-\sigma}{b}ns + \frac{\sigma}{\omega}(i-s) + \eta ns = \frac{b-\omega-\sigma}{b}ns + \frac{\sigma}{\omega}i + \left(\eta n - \frac{\sigma}{\omega}\right)s.$$
\end{proof}

The following result acts as a springboard from which to begin the proof of Lemma~\ref{expandorswaplemma}. 

\begin{prop} \label{findingacopy}
Set \(\eta > 0\) and \(m \in \mathbb{N}\), and let \(B\) be an \(r\)-partite bottle graph with neck \(\sigma\) and width \(\omega\), where \(b := |B|\). Define \(B^{*} := B(m)\). Then there exists \(n_0 \in \mathbb{N}\) such that the following holds. Suppose \(G\) is a graph on \(n \geq n_0\) vertices with degree sequence \(d_1 \leq d_2 \leq \ldots \leq d_n\) where 
\[d_i \geq \frac{b-\omega-\sigma}{b}n + \frac{\sigma}{\omega}i + \eta n \ \ \mbox{for all \(1 \leq i \leq \frac{\omega n}{b}\).}\]


\noindent Then there exists a copy of \(B^{*}\) in \(G\).
\end{prop}

\begin{proof}
Set \(\Delta := \Delta (B^{*})\). Let \(n\) be sufficiently large and define constants \(\varepsilon, d > 0\) and \(M' \in \mathbb{N}\) such that \(0 < 1/n \ll 1/M' \ll \varepsilon \ll d \ll 1/b, \eta, 1/\Delta\). 
Let $G$ be an $n$-vertex graph as in the statement of the proposition.
Applying Lemma~\ref{degformreglemma} with parameters \(\varepsilon\), \(d\) and \(M'\) to \(G\), we obtain clusters \(V_1, \ldots, V_k\), an exceptional set \(V_0\) and a pure graph \(G'\). Set \(q := |V_1| = \ldots = |V_k|\). 
Let \(R\) be the reduced graph of \(G\) with parameters \(\varepsilon\), \(d\) and \(M'\), where \(k := |R|\). 
By Lemma~\ref{inheriteddegseqreduced} we have that \(R\) has degree sequence \(d_{R,1} \leq d_{R,2} \leq \ldots \leq d_{R,k}\) where

\begin{equation} \label{Rdegseq} d_{R,i} \geq \frac{b-\omega-\sigma}{b}k + \frac{\sigma}{\omega}i + \frac{\eta k}{2} \ \ \mbox{for all} \ \ 1 \leq i \leq \frac{\omega k}{b}. \end{equation} By doing the following steps, we find a set \(\{x_1, \ldots, x_r\} \subseteq V(R)\) such that \(\{x_1, \ldots, x_r\}\) induces a copy of \(K_r\) in \(R\):\\

\noindent Step \(1\): Choose a vertex \(x_1 \in V(R)\) such that \[d_R(x_1) \geq k - \frac{\omega}{b}k + \frac{\eta k}{3}.\]

\noindent Such a vertex exists by (\ref{Rdegseq}).\\
    
\noindent Step \(i\) for each \(i \in \{2, \ldots, r-1\}\): We have that \(\{x_1, x_2, \ldots, x_{i-1}\}\) induces a copy of \(K_{i-1}\) in \(R\) and \[d_R(x_1), d_R(x_2), \ldots, d_R(x_{i-1}) \geq k - \frac{\omega}{b}k + \frac{\eta k}{3}.\] Let \(N_R(x_1, x_2, \ldots, x_{i-1}) := N_R(x_1) \cap N_R(x_2) \cap \ldots \cap N_R(x_{i-1})\). Then 
    
\begin{align*}|N_R(x_1, x_2,  \ldots, x_{i-1})| &\geq  k - \frac{(i-1)\omega}{b}k + \frac{(i-1)\eta k}{3}\\ &\geq  \frac{b - (r-2)\omega}{b}k + \frac{(i-1)\eta k}{3} = \frac{\omega + \sigma}{b}k + \frac{(i-1)\eta k}{3}.\end{align*}
    
\noindent 
Here the last equality follows as $b=\sigma +(r-1)\omega$.
Hence by (\ref{Rdegseq}) there exists \(y \in N_R(x_1, x_2, \ldots, x_{i-1})\) such that \(d_R(y) \geq k - \frac{\omega}{b}k + \frac{\eta k}{3}\). Let \(x_i := y\). \\
    
\noindent Step \(r\): We have that \(\{x_1, x_2, \ldots, x_{r-1}\}\) induces a copy of \(K_{r-1}\) in \(R\). Moreover, \[|N_R(x_1, x_2,  \ldots, x_{r-1})| \geq \frac{\sigma}{b}k + \frac{(r-1)\eta k}{3}.\]
Choose \(x_r\) to be any vertex in \(N_R(x_1, x_2,  \ldots, x_{r-1})\).

Therefore there exists a copy of \(K_r\) in \(R\), which implies that there exists a copy of \(B^{*}\) in \(R(\omega m)\). By Lemma~\ref{keylem} we have that there exists a copy of \(B^{*}\) in \(G\).
\end{proof}

 A crucial tool in the proof of Theorem~\ref{almostmainerror} is Lemma~\ref{expandorswaplemma} below. Before stating this lemma, we need two more definitions.

\begin{define} \label{expandingsetdef}
\textnormal{Set \(\ell \in \mathbb{N}\). Let \(G\) be a graph on \(n\) vertices and \(B\) be a bottle graph with neck \(\sigma\) and width \(\omega\). Suppose that there exists a \(B\)-tiling \(T\) of \(G\) and let \(\{z_1, \ldots, z_{\ell}\} \subseteq V(G)\setminus V(T)\). We say that \(\{z_1, \ldots, z_{\ell}\}\) is an \textit{expanding set of size \(\ell\) for \(T\) in \(G\)} if the following is true: there exists an injection \(f: \{z_1, \ldots, z_{\ell}\} \to T\) such that \(z_i\) has a neighbour in every \(\omega\)-vertex class of \(f(z_i)\) for each \(1 \leq i \leq \ell\).}
\end{define}


\begin{define} \label{swappingsetdef}
\textnormal{Set \(k, \ell, m \in \mathbb{N}\). Let \(G\) be a graph on \(n\) vertices and let \((v_1, v_2, \ldots, v_n)\) be an  ordering of the vertices of \(G\). 
Let \(B\) be a bottle graph with neck \(\sigma m\) and width \(\omega m\). Suppose that there exists a \(B\)-tiling \(T\) of \(G\) and let \(\{z_1, \ldots, z_{\ell}\} \subseteq V(G)\setminus V(T)\). 
Denote by \(\Omega_T\) the set of all vertices in \(V(G)\) that belong to \(\omega m\)-classes of copies of \(B\) in \(T\). 
Let \(z \in \{z_1, \ldots, z_{\ell}\}\) and \(y \in \Omega_T\), and denote by \(B_y\) the copy of \(B\) in \(T\) that contains \(y\). 
Then there exist \(1 \leq i,j \leq n\) such that \(z := v_i\), \(y := v_j\) and \(i \neq j\).  We say that \((z,y)\) is a \textit{\(k\)-swapping pair with respect to $(v_1,\dots, v_n)$} 
if the following is true: \(z\) is adjacent to at least \(\sigma\) vertices in the \(\sigma m\)-class of \(B_y\); \(z\) is adjacent to at least \(\omega\) 
vertices in each \(\omega m\)-class of \(B_y\) that does not contain \(y\); and \(j \geq i + k\). We say that \(\{z_1, \ldots, z_{\ell}\}\) is a \textit{\(k\)-swapping set 
of size \(\ell\) for \(T\) in \(G\) with respect to $(v_1,\dots, v_n)$} if there exists a set of $\ell$ vertices \(\{y_1, \ldots, y_{\ell}\} \subseteq \Omega_T\) such that \((z_i, y_i)\) is a \(k\)-swapping pair with respect to $(v_1,\dots, v_n)$
for each \(1 \leq i \leq \ell\) and \(B_{y_p} \neq B_{y_q}\) for all \(p \neq q\).}
\end{define}

Suppose $\mathcal B$ is a $B$-tiling in a reduced graph $R$.
Very roughly speaking the purpose of expanding sets is to extend  $\mathcal B$ to a larger tiling whilst swapping sets allow us to `rotate' which vertices are uncovered by our tiling (which helps for future expansion of $\mathcal B$ to a larger tiling).


\section{Almost perfect \(H\)-tilings in graphs}\label{almost}

\begin{lem} \label{expandorswaplemma}
Let \(B\) be an \(r\)-partite bottle graph with neck \(\sigma\) and width \(\omega\), where  \(b := |B|\). Set \(\eta, \gamma > 0\) and \(n, m \in \mathbb{N}\) such that \(0 < 1/n \ll \gamma \ll 1/m \ll \eta \ll 1/b\). Set \(B^{*} := B(m)\). Let \(G\) be a graph on \(n\) vertices with degree sequence \(d_1 \leq d_2 \leq \ldots \leq d_n\) where
\[d_i \geq \frac{b-\omega-\sigma}{b}n + \frac{\sigma}{\omega}i + \eta n \ \ \mbox{\textit{for all} \(1 \leq i \leq \frac{\omega n}{b}\).}\] 
Let $V(G)=\{v_1,\dots,v_n\}$ such that $d_G (v_i)=d_i$ for all $1\leq i \leq n$.
Suppose the largest \(B^{*}\)-tiling in \(G\) covers precisely \(n' \leq (1 - \eta)n\) vertices. Then for any \(B^{*}\)-tiling \(T\) covering \(n'\) vertices in \(G\) there exists an expanding set of size \(\gamma n\) for \(T\) in \(G\) or an \( \frac{\omega \gamma n }{\sigma}\)-swapping set of size \(\gamma n\) for \(T\) in \(G\) with respect to $(v_1,\dots,v_n)$.
\end{lem}


\begin{proof}
By repeatedly applying Proposition~\ref{findingacopy}, we see that \(n' \geq \eta n/2\). Define a bijection \(I: V(G) \to [n]\) where \(I(x) = i\) implies that \(d_{G}(x) = d_i\). Let \(V(G) := \{v_1, \ldots, v_n\}\) such that \(I(v_i) = i\). Set \(n'' := n - n'\) and let \(G'' := G\setminus V(T)\). Let \(V(G'') = \{x_1, \ldots, x_{n''}\}\) where \(I(x_1) < I(x_2) < \ldots < I(x_{n''})\). For each \(1 \leq i \leq n''\), set \(s_i := I(x_i)\). Then  \(d_G(x_i) = d_{s_i}\). Choose \(j\) to be the largest integer such that 
\[d_G(x_{j}) \leq \frac{b-\omega}{b}n + (\eta - 2\gamma)n.\] 
Notice that $s_j \leq \omega n/b$.
We will refer to \(x_1, \ldots, x_j\) as \textit{small} vertices and \(x_{j+1}, \ldots, x_{n''}\) as \textit{big} vertices. 

{\it Case 1:} Suppose we have \(\gamma n\) big vertices \(z_1, \ldots, z_{\gamma n} \in V(G'')\) such that 
\begin{align}\label{superlabel}
d_{G}(z_i, G'') \leq \frac{b-\omega}{b}n'' + \frac{\eta n}{4} \ \ \text{for all} \ \ 1 \leq i \leq \gamma n.  
\end{align}
Then
\[d_G(z_i, T) \geq \frac{b-\omega}{b}n' + \frac{\eta n}{4} \ \ \mbox{for all} \ \ 1 \leq i \leq \gamma n.\] Set \(\omega^{*} := \omega m\). For each \(1 \leq i \leq \gamma n\), we see that \(z_i\) can be adjacent to at most a \(\frac{b-\omega}{b}\)-proportion of the vertices in \(T\) without having a neighbour in each \(\omega^{*}\)-class of some copy of \(B^{*}\) in \(T\). Since \(\gamma \ll 1/m \ll \eta \ll 1/b\), 
for each \(1 \leq i \leq \gamma n\) there are at least
\[\frac{\left(\frac{b-\omega}{b}n' + \frac{\eta n}{4}\right) - \left(\frac{b-\omega}{b}n'\right)}{\omega^{*}} = \frac{\eta n}{4\omega^{*}} \geq \gamma n\] copies of \(B^{*}\) in \(T\) that have at least one neighbour of \(z_i\) in each of their \(\omega^{*}\)-classes. Thus we can define an injection \(f:\{z_1, \ldots, z_{\gamma n}\} \to T\) such that \(z_i\) has a neighbour in each \(\omega^{*}\)-class of \(f(z_i)\) for each \(1 \leq i \leq \gamma n\). Hence \(\{z_1, \ldots, z_{\gamma n}\}\) is an expanding set of size \(\gamma n\) for \(T\) in \(G\).

\smallskip

{\it Case 2:} We may assume there does not exist an expanding set of size \(\gamma n\) for \(T\) in \(G\). 

In particular, there are at most \(\gamma n - 1\) vertices in \(V(G'')\) that have a neighbour in every \(\omega^{*}\)-class of \(\gamma n\) copies of \(B^{*}\) in \(T\). (Note that these could be small or big vertices.) Remove such vertices from \(V(G'')\) and call the remaining graph \(G_1''\). 
In particular, no big vertex in $G''_1$ satisfies (\ref{superlabel}). Set \(n_1'' := |G_1''|\).

{\it Subcase A:} Suppose we have \(\gamma n\) small vertices \(x_{i_1}, \ldots, x_{i_{\gamma n}} \in V(G_1'')\) such that
\begin{equation}\label{impliesswappingset}
d_{G}(x_{i_\ell}, G_1'') \leq \frac{b-\omega-\sigma}{b}n_1'' + \frac{\sigma}{\omega}i_\ell + 2\gamma n \ \ \mbox{for all} \ \ 1 \leq \ell \leq \gamma n.
\end{equation} 
Then by (\ref{impliesswappingset}) and the degree sequence condition of the lemma, we have
\begin{equation}\label{swappingequation}
d_G(x_{i_\ell}, T) \geq \frac{b-\omega-\sigma}{b}n' + \left(\frac{\sigma}{\omega}s_{i_\ell} - \frac{\sigma}{\omega}i_\ell \right) + \frac{\eta n}{2} \ \ \mbox{for all} \ \ 1 \leq \ell \leq \gamma n.
\end{equation} 
Let \(k \in \{1, \ldots \gamma n\}\). Denote by \(\Omega^{*}_T\) the set of all vertices in \(G\) that belong to \(\omega^{*}\)-classes of copies of \(B^{*}\) in \(T\). Set \(\sigma^{*} := \sigma m\). 
We aim to count the number of vertices \(y \in \Omega_T^{*}\) such that \((x_{i_k},y)\) is an \( \frac{\omega \gamma n }{\sigma}\)-swapping pair (with respect to $(v_1,\dots,v_n)$). 
Let \(T_1\) denote the subcollection of copies \(B_1\) of \(B^{*}\) in \(T\) such that \(x_{i_k}\) is adjacent to a vertex in every \(\omega^*\)-class of \(B_1\). 
Then since we removed earlier all vertices that have a neighbour in every \(\omega^{*}\)-class of \(\gamma n\) copies of \(B^{*}\) in \(T\), we have 
\[d_G(x_{i_k}, T_1) \leq (\gamma n - 1)bm.\] 

Suppose \(y \in \Omega^{*}_T\) and let \(B^{*}_{y}\) be the copy of \(B^{*}\) in \(T\) containing \(y\). We say \(y\) is \textit{swappable} with \(x_{i_k}\) 
if \(x_{i_k}\) is adjacent to at least \(\sigma\) vertices in the \(\sigma^{*}\)-class of \(B^{*}_y\) and at least \(\omega\) vertices in each \(\omega^{*}\)-class of \(B^{*}_y\) 
that does not contain \(y\). Denote the set of vertices that are swappable with \(x_{i_k}\) by \(S(x_{i_k})\). Let \(T_2\) denote the subcollection of copies \(B_2\) of \(B^{*}\) in 
\(T\setminus T_1\) such that \(B_2\) does not contain any vertex in \(S(x_{i_k})\). Then 
\[d_G(x_{i_k}, T_2) \leq (bm - \omega m  - \sigma m + \sigma - 1)|T_2|.\] Note that the \(- \omega m\) term is present since \(x_{i_k}\) cannot be adjacent to a vertex in every 
\(\omega^{*}\)-class of any copy of \(B^*\) in \(T\setminus T_1\). Let \(T_3 := T\setminus (T_1 \cup T_2)\). Then 
\[d_G(x_{i_k}, T_3) \leq (bm - \omega m)|T_3|.\] Observe that \(|T_1| + |T_2| + |T_3| = n'/bm\). Then 
\begin{equation}\label{intotequation}
d_G(x_{i_k}, T) = d_G(x_{i_k}, T_1) + d_G(x_{i_k}, T_2 \cup T_3) \leq (\gamma n - 1)bm + \left(\frac{b-\omega-\sigma}{b}n' + \frac{\left(\sigma - 1\right)}{bm}n'\right) + \sigma m|T_3|.
\end{equation} 
 Using  (\ref{swappingequation}) and (\ref{intotequation}) we see that 
\[|T_3| \geq \frac{s_{i_k} - i_k}{\omega m} + \frac{\eta n}{2\sigma m}
-\frac{(\gamma n-1)b}{\sigma} - \frac{(\sigma-1)n'}{b\sigma m^2}
 \geq \frac{s_{i_k} - i_k}{\omega m} + \frac{\eta n}{8\sigma m},\] 
where the the last inequality follows as $\gamma \ll 1/m \ll \eta \ll 1/b$.

Note that as $T_3 \cap T_2 = \emptyset$,
 every copy \(B_3\) of \(B^{*}\) in \(T_3\) must contain a vertex
from \(S(x_{i_k})\). By definition of swappable, this in fact implies
 that every copy \(B_3\) of \(B^{*}\) in \(T_3\) must contain \(\omega^{*}\) 
vertices from \(S(x_{i_k})\). Hence there are at least \(s_{i_k} - i_k + \frac{\omega\eta n}{8\sigma}\) vertices in \(S(x_{i_k})\). Not all vertices in \(S(x_{i_k})\) 
may form an \(\omega \gamma n/\sigma\)-swapping pair with \(x_{i_k}\) (with respect to $(v_1,\dots, v_n)$). 
Indeed, there are at most \(s_{i_k} - i_k + \frac{\omega\gamma n}{\sigma}\) vertices \(y \in S(x_{i_k})\) 
with \(I(y) < s_{i_k} + \frac{\omega \gamma n }{\sigma} \)
(and so do not form an \(\omega \gamma n/\sigma\)-swapping pair with \(x_{i_k}\)).
 Hence, since \(\gamma \ll 1/m,\eta,1/b\), there are at least 
\[\frac{\omega \eta n}{16\sigma} \geq bm\gamma n\] 
vertices \(y \in \Omega_T^{*}\) 
such that \((x_{i_k},y)\) is an \(\frac{\omega \gamma n }{\sigma}\)-swapping pair. 
Therefore, since \(k \in \{1, \ldots, \gamma n\}\) was arbitrary, for each \(\ell \in \{1, \ldots, \gamma n\}\) 
there exist at least \(bm\gamma n\) vertices \(y \in \Omega_T^{*}\) such that \((x_{i_{\ell}},y)\) is an \(\frac{\omega \gamma n }{\sigma}\)-swapping pair. 
Hence there exists a set of vertices \(\{y_1, \ldots, y_{\gamma n}\} \subseteq \Omega^{*}_T\) such that \((x_{i_{\ell}}, y_{\ell})\) is an \( \frac{\omega \gamma n }{\sigma}\)-swapping pair 
for each \(1 \leq \ell \leq \gamma n\) and \(B^{*}_{y_i} \neq B^{*}_{y_j}\) for all \(i \neq j\).  Thus \(\{x_{i_1}, \ldots, x_{i_{\gamma n}}\}\) is an \( \frac{\omega \gamma n }{\sigma}\)-swapping set  of size \(\gamma n\) for \(T\) in \(G\).

{\it Subcase B:}
Assume there does not exist an \(\frac{\omega \gamma n }{\sigma}\)-swapping set of size \(\gamma n\) for \(T\) in \(G\).

Then there are at most \(\gamma n - 1\) small vertices \(x \in V(G_1'')\) that satisfy (\ref{impliesswappingset}). Remove such vertices from \(V(G_1'')\), call the remaining graph \(G_2''\) and set \(n_2'' := |G_2''|\). Then for every small vertex \(x_i \in V(G_2'')\) we have 
\[d_{G_2''}(x_i) \geq \frac{b-\omega-\sigma}{b}n_2'' + \frac{\sigma}{\omega}i + \gamma n_2''.\] 
For every big vertex \(y \in V(G_2'')\), recall that $y$ does not satisfy (\ref{superlabel}). So since $|G''\setminus G''_2|\leq 2 \gamma n$, we have
\[d_{G_2''}(y) \geq \frac{b-\omega}{b}n_2'' + \gamma n_2''.\] 
Thus, 
\(G_2''\) has degree sequence \(d_{G_2'',1} \leq d_{G_2'',2} \leq \ldots \leq d_{G_2'',n_2''}\) such that
\[d_{G_2'', i} \geq \frac{b-\omega-\sigma}{b}n_2'' + \frac{\sigma}{\omega}i + \gamma n_2'' \ \ \mbox{for all} \ \ 1 \leq i \leq \frac{\omega n_2''}{b}.\] 
Hence, by Proposition~\ref{findingacopy} there exists a copy of \(B^{*}\) in \(G_2''\), contradicting that the largest \(B^{*}\)-tiling in \(G\) covers \(n'\) vertices.

Thus Lemma~\ref{expandorswaplemma} holds.
\end{proof}

With Lemma~\ref{expandorswaplemma} at hand we now can prove Theorem~\ref{almostmainerror}.
\begin{prooff} 
If \(\sigma = \omega\), then Theorem~\ref{almostmainerror} is equivalent to the (non-directed) graph version of \cite[Theorem 4.2]{Tregs}.

So we may assume that  \(\sigma < \omega\). Set \(\sigma' := (r-1)\sigma\) and \(\omega' := (r-1)\omega\). Let \(B\) be the $r$-partite bottle graph with neck \(\sigma'\) and width \(\omega'\), set \(b := |B|\) and observe that \(B\) has a perfect \(H\)-tiling. 
Let $t:=(\omega '-\sigma ')b$.
Note that it  suffices to prove the theorem under the additional assumption that $\eta \ll 1/b$.
 Define additional constants \(\varepsilon, d, \gamma \in \mathbb{R}\) and \(M',m \in \mathbb{N}\) such that 
$$0 < 1/n \ll 1/M' \ll \varepsilon \ll d \ll \gamma \ll 1/m \ll \eta \ll 1/b.$$ 

Let \(B^{*} := B(m)\) and set \[S := \frac{2\sigma}{\omega\gamma^2}, \ Q := \lceil 1/\gamma \rceil \ \ \mbox{and} \ \ z:= Q(S+1).\] 
Note that \(B^{*}\) has a perfect \(H\)-tiling. 

Suppose $G$ is an $n$-vertex graph as in the statement of the theorem.
Apply Lemma~\ref{degformreglemma} with parameters \(\varepsilon\), \(d\) and \(M'\) to \(G\). This gives us clusters \(V_1, \ldots, V_k\), an exceptional set \(V_0\) and a pure graph \(G'\), where \(|V_0| \leq \varepsilon n\) and \(|V_1| = \ldots = |V_k| =: q\). Let \(R\) be the reduced graph of \(G\) with parameters \(\varepsilon, d\) and \(M'\); so $k = |R|$. By Lemma \nolinebreak \ref{inheriteddegseqreduced}, \(R\) has degree sequence \(d_{R,1} \leq d_{R,2} \leq \ldots \leq d_{R,k}\) such that

\begin{equation*} 
d_{R,i} \geq \left(1-\frac{\omega+\sigma}{h}\right)k + \frac{\sigma}{\omega}i + \frac{\eta k}{2}=\left(1-\frac{\omega '+\sigma'}{b}\right)k + \frac{\sigma '}{\omega '}i + \frac{\eta k}{2}  \ \ \mbox{for all \(1 \leq i \leq \frac{\omega k}{h}=\frac{\omega ' k}{b}\).}
\end{equation*}

In what follows when we consider an $s$-swapping set in some blow-up $R(w)$ of $R$, we always implicitly mean an $s$-swapping set in $R(w)$ with respect to $(v_1,\dots, v_{kw})$ where 
$V(R(w))=\{ v_1, \dots, v_{kw}\}$ and $d_{R(w)} (v_1)\leq d_{R(w)} (v_2)\leq \dots \leq d_{R(w)} (v_{kw})$. That is, each blow-up $R(w)$ of $R$ comes equipped with an  ordering of its vertices based on the degrees; these orderings are defined by the functions $I_j$ below.
\begin{claim}\label{almostperfectclaim}
\(R' := R((mt)^z)\) contains a \(B^{*}\)-tiling \(\mathcal{T}\) covering at least \((1 - \eta/2)k(mt)^z = (1-\eta/2)|R'|\) vertices.
\end{claim}

{\noindent \it Proof of Claim~\ref{almostperfectclaim}.}
If \(R\) contains a \(B^{*}\)-tiling covering at least \((1-\eta/2)k\) vertices then Lemma~\ref{mtblowuplem} implies that Claim~\ref{almostperfectclaim} holds. Suppose then that the largest \(B^{*}\)-tiling \(T\) in \(R\) covers exactly $c$ vertices where \(c < (1-\eta/2)k\). Then by Lemma~\ref{expandorswaplemma}, there exists an expanding set of size \(\gamma k\) for \(T\) in \(R\) or an \( \frac{\omega \gamma k }{\sigma}\)-swapping set of size \(\gamma k\) for \(T\) in \(R\). Define \(B'\) to be the \(r\)-partite bottle graph with neck \(\sigma'\) and width \(\omega' - 1\). Set \(\omega^{*} := \omega' m\). 

{\it Step~1: Find a \(B^{*}\)-tiling covering at least \((c+\gamma k)(mt)^{S+1}\) vertices in \(R((mt)^{S+1})\).}

{\bf Case~1:}
There exists an expanding set \(\{z_1, \ldots, z_{\gamma k}\}\) for \(T\), and hence also an associated injection \(f: \{z_1, \ldots, z_{\gamma k}\} \to T\).

In this case we do the following: For each \(1 \leq i \leq \gamma k\), separate \(R[z_i \cup f(z_i)]\) into a copy of \(K_r\) (containing \(z_i\) and one vertex from each \(\omega^{*}\)-class of \(f(z_i)\)), a copy of \(B'\) and a copy of \(B(m-1)\). Then we have a \((B^{*}, B(m-1), B', K_r)\)-tiling in \(R\) covering at least \(c + \gamma k\) vertices.  By Lemma~\ref{mtblowuplem}, \(R(mt)\) contains a \(B^{*}\)-tiling covering at least \((c +\gamma k)mt\) vertices. Further applying  Lemma~\ref{mtblowuplem} we obtain   a \(B^{*}\)-tiling covering at least \((c+\gamma k)(mt)^{S+1}\) vertices in \(R((mt)^{S+1})\), as desired.

{\bf Case~2:} There does not exist an expanding set of size $\gamma k$ for $T$ in
$R$.

For each \(1 \leq j \leq S\), Proposition~\ref{inheriteddegseqG(s)} implies that \(R((mt)^j)\) has degree sequence \(d_{R((mt)^j), 1} \leq d_{R((mt)^j), 2} \leq \ldots \leq d_{R((mt)^j), k(mt)^j}\) such that
\[d_{R((mt)^j), i} \geq \left(1-\frac{\omega+\sigma}{h}\right)k(mt)^j + \frac{\sigma}{\omega}i + \left(\frac{\eta k}{2} - \frac{\sigma}{\omega}\right)(mt)^j \ \ \mbox{for all \(1 \leq i \leq \frac{\omega k(mt)^j}{h}\).}\]

Define for \(0 \leq j \leq S\) bijections \(I_j: V(R((mt)^j)) \to [k(mt)^j]\) where \(I_j(x) := i\) implies that \(d_{R((mt)^j)}(x) = d_{R((mt)^j),i}\). 
In particular, suppose that $x \in V(R)$ and let $x_1,\dots, x_{(mt)^j}$ denote the $(mt)^j$ vertices in $R((mt)^j)$ that correspond to $x$. Suppose that  $I_0 (x)=i$.
Then we may assume that 
\begin{align}\label{label2}
I_j (x_s)= (i-1)(mt)^j+s > (I_0(x)-1)(mt)^j \text{ for each } 1\leq s \leq (mt)^j.
\end{align}
To put all this another way, one can view $I_0$ as an ordering of the vertices in $R$ in terms of the vertex degrees; $I_j$ is the ordering of $R((mt)^j)$ `inherited' from the ordering $I_0$.

Note that for all $0 \leq j \leq S$,
\begin{align}\label{label3}
\left(\sum_{x \in V(R((mt)^j))} I_j(x)\right) \leq k^2 (mt)^{2j}. 
\end{align}

 Denote by \(\Omega_T^{*}\) the set of all vertices in \(V(R)\) that belong to \(\omega^{*}\)-classes of copies of \(B^{*}\) in \(T\). 
As there does not exist an expanding set of size \(\gamma k\) for \(T\) in \(R\), then there exists an \( \frac{\omega \gamma k }{\sigma}\)-swapping set \(\{z_1, \ldots, z_{\gamma k}\}\) for \(T\) in \(R\). 
Hence there also exists a set \(\{y_1, \ldots, y_{\gamma k}\} \subseteq \Omega^{*}_T\) such that \((z_i, y_i)\) is an \( \frac{\omega \gamma k }{\sigma}\)-swapping pair for each \(1 \leq i \leq \gamma k\), such that \(B^{*}_{y_i} \neq B^{*}_{y_j}\)\footnote{As in Definition~\ref{swappingsetdef}.} for all \(i \neq j\), and such that \(I_0(y_i) \geq I_0(z_i) + \frac{\omega \gamma k }{\sigma}\) for all \(1 \leq i \leq \gamma k\). 

 For each \(1\leq i \leq \gamma k\), note that  \(R[(z_i \cup V(B^{*}_{y_i}))\setminus \{y_i\}]\) can be separated into a copy of \(B\) containing \(z_i\) and a copy of \(B(m-1)\). 
Then we have a \((B^{*}, B(m-1), B)\)-tiling \(T_1\) covering \(c\) vertices in \(R\). 
Further, since each \((z_i, y_i)\) is an \( \frac{\omega \gamma k }{\sigma}\)-swapping pair, we have that
\begin{align}\label{label1}
\left(\sum_{x \in V(R)\setminus V(T_1)} I_0(x)\right) \geq \left(\sum_{x \in V(R)\setminus V(T)} I_0(x)\right)+\frac{\omega \gamma ^2 k^2}{\sigma}.
\end{align}

By Lemma~\ref{mtblowuplem}, \(T_1(mt) \) contains a perfect  \(B^{*}\)-tiling $T'$, i.e. $T'$ is a  \(B^{*}\)-tiling covering \(c(mt)\) vertices in \(R(mt)\).
Observe that \(T'\) in \(R(mt)\) covers proportionally the same amount of vertices as \(T\) in \(R\).
Further, (\ref{label2}) and (\ref{label1}) imply that
 \begin{align}
\nonumber
\sum_{x \in V(R(mt))\setminus V(T')} I_1(x) & \geq \left ( \sum_{x \in V(R)\setminus V(T_1)} (I_0(x)-1) \right ) (m t)^2 \\ \label{label4}
&  \geq \left(\left(\sum_{x \in V(R)\setminus V(T)} I_0(x)\right) + \frac{\omega(\gamma k)^2}{2\sigma}\right)(mt)^2. 
\end{align}

Denote by \(\Omega_{T'}^{*}\) the set of all vertices in \(R(mt)\) that belong to \(\omega^{*}\)-classes of copies of \(B^{*}\) in \(T'\). Suppose that there does not exist an expanding set of size \(\gamma kmt\) for \(T'\) in \(R(mt)\). 
Then by Lemma~\ref{expandorswaplemma} there must exist an \(\frac{\omega\gamma kmt}{\sigma}\)-swapping set of size \(\gamma kmt\) for \(T'\) in \(R(mt)\). As before we can produce a \mbox{\((B^{*}, B(m-1), B)\)-tiling} \(T'_1\) covering \(c(mt)\) vertices 
in \(R(mt)\). 
Then by Lemma~\ref{mtblowuplem}, \(T'_1(mt)\)
contains  a 
perfect \(B^{*}\)-tiling $T''$, i.e. $T''$ is a $B^*$-tiling covering \(c(mt)^2\) vertices in \(R((mt)^2)\).
 Observe, similarly as before, that \(T''\) in \(R((mt)^2)\) covers proportionally the same amount of vertices as \(T\) in \(R\) and 
\begin{align*} 
\sum_{x \in V(R((mt)^2))\setminus V(T'')} I_2(x) & 
\geq \left ( \left(\sum_{x \in V(R(mt))\setminus V(T')} I_1(x)\right) + \frac{\omega(\gamma kmt)^2}{2\sigma}\right)(mt)^2 \\
\\ & \stackrel{(\ref{label4})}{ \geq} \left(\left(\sum_{x \in V(R)\setminus V(T)} I_0(x)\right) + \frac{\omega(\gamma k)^2}{\sigma}\right)(mt)^4.
\end{align*}

Note that (\ref{label3}) implies that one can
repeat this argument at most \(S\) times; that is, for some $j \leq S$ we must obtain an expanding set of size  \(\gamma k(mt)^{j}\)  in \(R((mt)^{j})\). More precisely,
we  obtain a $B^*$-tiling \(T^{(j)}\) in \(R((mt)^{j})\) covering $c(mt)^j$ vertices, such that
 there  exists an expanding set of size \(\gamma k(mt)^{j}\) for \(T^{(j)}\) in \(R((mt)^{j})\). Then as before, one can use this expanding set and 
 Lemma~\ref{mtblowuplem} to obtain   a \(B^{*}\)-tiling covering at least \((c+\gamma k)(mt)^{S+1}\) vertices in \(R((mt)^{S+1})\), as desired.

\smallskip

{\noindent \it General Step:}

Repeating the whole argument from Step~1 at most \(Q\) times we see that \(R((mt)^{Q(S+1)}) = R((mt)^z) = R'\) has a \(B^{*}\)-tiling \(\mathcal{T}\) covering at least \((1 - \eta/2)|R'|\) vertices. Thus Claim \ref{almostperfectclaim} holds.
\qed

\medskip

Now for each \(1 \leq i \leq k\), partition \(V_i\) into classes \(V_i^{*}, V_{i,1}, \ldots, V_{i, (mt)^z}\) where \(q' := |V_{i,j}| = \lfloor q/(mt)^z\rfloor \geq q/(2(mt)^{z})\) for all \(1 \leq j \leq (mt)^z\). Lemma~\ref{degformreglemma} implies that \(qk \geq (1 - \varepsilon)n\), therefore 
\begin{align}\label{label10}
q'|R'| = \lfloor q/(mt)^z\rfloor k(mt)^z \geq qk - k(mt)^z \geq (1 - 2\varepsilon)n.
\end{align}
Fact~\ref{slicinglemma} tells us that for each \(\varepsilon\)-regular pair \((V_{i_1}, V_{i_2})_{G'}\) with density at least \(d\) we have that \((V_{i_1, j_1}, V_{i_2, j_2})_{G'}\) is \(2\varepsilon (mt)^z\)-regular with density at least \(d-\varepsilon \geq d/2\) (for all \(1 \leq j_1, j_2 \leq (mt)^z\)). 
Note that $2\varepsilon (mt)^z\leq \varepsilon ^{1/2}$.
So we can label the vertex set of \(R'\) so that \(V(R') = \{V_{i,j} : 1 \leq i \leq k, 1 \leq j \leq (mt)^z\}\) and see that if 
\(V_{i_1, j_1}V_{i_2, j_2} \in E(R')\) then \((V_{i_1, j_1}, V_{i_2, j_2})_{G'}\) is \(\varepsilon ^{1/2}\)-regular with density at least \(d/2\).

We know by Claim~\ref{almostperfectclaim} that \(R'\) has a \(B^{*}\)-tiling \(\mathcal{T}\) that covers at least \((1-\eta/2)|R'|\) vertices. Let \(\hat{B}^{*}\) be a copy of \(B^{*}\) in \(\mathcal{T}\) and label the vertices of \(\hat{B}^{*}\) so that \(V(\hat{B}^{*}) = \{V_{i_1, j_1}, V_{i_2, j_2}, \ldots, V_{i_{bm}, j_{bm}}\}\). Set \(V' := V_{i_1, j_1} \cup V_{i_2, j_2} \cup \ldots \cup V_{i_{bm}, j_{bm}}\). 
Applying Lemma~\ref{almosttilingRtoG} with  \(\eta ^2, q', d/2, \varepsilon ^{1/2}\) playing the roles of $\beta, q,d, \varepsilon$, we have that \(G'[V']\) has a \(B^{*}\)-tiling covering at least \((1-\eta ^2)q'bm\) vertices. Applying Lemma~\ref{almosttilingRtoG} in this way to each copy of \(B^{*}\) in \(\mathcal{T}\) we see that \(G' \subseteq G\) has a \(B^{*}\)-tiling covering at least
\[
\left(\left(1-\eta ^2\right)q'bm\right) \times \left(\left(1-\eta/2\right)|R'|\right)/bm \stackrel{(\ref{label10})}{\geq} \left(1-\eta ^2\right)\left(1-\eta/2\right)\left(1-2\varepsilon\right)n \geq (1-\eta)n
\] 
vertices. Since each copy of \(B^{*}\) has a perfect \(H\)-tiling, \(G\) contains an \(H\)-tiling covering all but at most \(\eta n\) vertices.  \endproof
\end{prooff}

Theorem~\ref{almostmain} easily implies Theorem~\ref{xdegseqKomlos}.

\begin{prooffff}
Let $H$, $x \in (0,1)$ and $\eta >0$ be as in the statement of the theorem. 
Suppose $n$ is sufficiently large and let $G$ be an $n$-vertex graph as in the statement of the theorem.

Note that it suffices to prove the result in the case when  
 \(x \in (0,1) \cap \mathbb{Q}\). Thus, there exist \(a,b \in \mathbb{N}\) such that \(x = a/b\). Define \(\sigma_1 := a(r-1)\sigma\) and \(\omega_1 := a(r-1)\omega + (b-a)h=bh-a \sigma\). Let \(H_1\) be the \(r\)-partite bottle graph with neck \(\sigma_1\) and width \(\omega_1\), and observe that $\sigma _1 <\omega _1$ and \(|H_1| = b(r-1)h\). 

\begin{claim*}
\(H_1\) contains an \(H\)-tiling covering \(x|H_1|\) vertices.
\end{claim*} 
The claim follows since one can tile $H_1$ with \(a(r-1)\) copies of \(H\) where each $\sigma$-class lies in the $\sigma _1$-class of $H_1$. Thus, we have an \(H\)-tiling covering \(a(r-1)h = x|H_1|\) vertices in \(H_1\), as desired.

Note that
$$d_i \geq \left(g_H(x) - \frac{x\sigma}{h}\right)n + \frac{(r-1)x\sigma}{h - x\sigma}i= 
\left ( 1-\frac{\omega _1+\sigma _1}{ b(r-1)h} \right ) n+\frac{\sigma _1}{\omega _1}i $$
for all $i \leq \left(\frac{h - x\sigma}{(r-1)h}\right)n =\frac{\omega _1 n}{b(r-1)h}$.
Thus, applying Theorem~\ref{almostmain} with \(H_1\)  playing the role of $H$, we produce an \(H_1\)-tiling in $G$ covering all but at most \(\eta n\) vertices. Then the claim implies that we have an \(H\)-tiling in \(G\) covering at least \(x(1 - \eta ) n > (x-\eta)n\) vertices.
\endproof
\end{prooffff}

\section{Concluding remarks}

In this paper we have given a particular degree sequence condition that forces a graph to contain an almost perfect \(H\)-tiling (Theorem \ref{almostmain}). In fact, 
in general for a fixed graph $H$,
Theorem~\ref{almostmain} yields a whole class of degree sequences that force an almost perfect \(H\)-tiling. Indeed, we have the following consequence of Theorem~\ref{almostmain}.   

\begin{thm}\label{generalalmostmain}
Let \(\eta > 0\) and \(H\) be a graph with \(\chi(H) = r\) and \(h := |H|\). Set \(\sigma \in \mathbb{R}\) such that \(\sigma(H) \leq \sigma \leq h/r\) and \(\omega := \left(h - \sigma\right)/(r-1)\). Then there exists an \(n_0 = n_0(\eta, \sigma, H) \in \mathbb{N}\) such that the following holds: Suppose \(G\) is a graph on \(n\geq n_0\) vertices with degree sequence \(d_1 \leq d_2 \leq \ldots \leq d_n\) such that

\[d_i \geq \left(1 - \frac{\omega+\sigma}{h}\right)n + \frac{\sigma}{\omega}i \ \ \mbox{for all \ \(1 \leq i \leq \frac{\omega n}{h}\).}\]
Then \(G\) contains an \(H\)-tiling covering all but at most \(\eta n\) vertices.
\end{thm}

\begin{proof} 
Note that it suffices to prove the theorem under the assumption that \(\sigma \in \mathbb{Q}\).
To prove Theorem \ref{generalalmostmain}, we define a certain bottle graph \(H^*\) and then apply Theorem \ref{almostmain} with input \(H^*\) to conclude our result. 

Since \(\sigma \in \mathbb{Q}\), there exist \(a,b \in \mathbb{N}\) such that \(\sigma = a/b\). Let \(\omega(H) := (h - \sigma(H))/(r-1)\) and \(t := b(r-1)(\omega(H) - \sigma(H))\). We define \(H^*\) to be the \(r\)-partite bottle graph with neck \(\sigma t\) and width \(\omega t\) (note \(\sigma t, \omega t \in \mathbb N\)). Also, notice that \(|H^*| = ht\). 

\begin{claim*}
\(H^*\) contains a perfect \(H\)-tiling.
\end{claim*}

We tile \(t\) copies of \(H\) into \(H^*\). Firstly, tile \(b(r-1)(\omega(H) - \sigma)\) copies of \(H\) into \(H^*\) such that the \(\sigma(H)\)-classes are all placed in the \(\sigma t\)-class of \(H^*\). This leaves 
\[\sigma b(r-1)(\omega(H) - \sigma(H)) - \sigma(H) b (r-1)(\omega(H) - \sigma) = \omega(H)b(r-1)(\sigma - \sigma(H))\] 
vertices in the \(\sigma t\)-class of \(H^*\) to be covered and 
\begin{align*}
& \ \omega b (r-1)(\omega(H) - \sigma(H)) - \omega(H) b(r-1)(\omega(H) - \sigma) \\ 
= & \ b((r-1)\omega(\omega(H) - \sigma(H)) - (r-1)\omega(H)(\omega(H) - \sigma)) \\
= & \ b((h - \sigma)(\omega(H) - \sigma(H)) - (h - \sigma(H))(\omega(H) - \sigma)) \\
= & \ b(h - \omega(H))(\sigma - \sigma(H))
\end{align*}
vertices in each \(\omega t\)-class of \(H^*\) to be covered. Let \(\overline{H}\) be the \(r\)-partite complete graph with one vertex class of size \((r-1)\omega(H)\) and \((r-1)\) vertex classes of size \((r-2)\omega(H) + \sigma(H)\). Observe that \(\overline{H}\) has a perfect \(H\)-tiling (using \(r-1\) copies of \(H\)). To cover the remaining vertices of \(H^*\), tile \(b(\sigma - \sigma(H))\) copies of \(\overline{H}\) into \(H^*\) such that every vertex class of size \((r-1)\omega(H)\) is placed in the \(\sigma t\) class of \(H^*\). Observe that \[((r-2)\omega(H) + \sigma(H))b(\sigma - \sigma(H)) = b(h - \omega(H))(\sigma - \sigma(H)).\] Hence \(H^*\) contains a perfect \(H\)-tiling and the claim holds.

\smallskip

Suppose $G$ is as in the statement of Theorem~\ref{generalalmostmain}.
Applying Theorem \ref{almostmain} with input $G$ and \(H^*\), we obtain that \(G\) contains an \(H^*\)-tiling covering all but at most \(\eta n\) vertices. 
(Note the degree sequence in Theorem~\ref{generalalmostmain} is precisely the degree sequence of Theorem~\ref{almostmain} with input \(H^*\).)
Since each copy of \(H^*\) has a perfect \(H\)-tiling, we conclude that \(G\) contains an \(H\)-tiling covering all but at most \(\eta n\) vertices.   

\end{proof}

In a similar way, Theorem \ref{xdegseqKomlos} yields a class of degree sequences forcing an almost \(x\)-proportional \(H\)-tiling in \(G\).

\begin{thm}\label{generalxdegseqKomlos}
Let \(x \in (0,1)  \) and \(H\) be a graph with \(\chi(H) = r\) and \(h := |H|\). Set \(\eta > 0\). Let \(\sigma \in \mathbb{R}\) such that \(\sigma(H) \leq \sigma \leq h/r\) and \(\omega := \left(h - \sigma\right)/(r-1)\). Then there exists an \(n_0 = n_0(\eta,x,\sigma, H) \in \mathbb{N}\) such that the following holds: Suppose \(G\) is a graph on \(n\geq n_0\) vertices with degree sequence \(d_1 \leq d_2 \leq \ldots \leq d_n\) such that

\[d_i \geq \left(g_H(x) - \frac{x\sigma}{h}\right)n + \frac{(r-1)x\sigma}{h - x\sigma}i \ \ \mbox{for all \ \(1 \leq i \leq \left(\frac{h - x\sigma}{(r-1)h}\right)n\).}\]
 Then \(G\) contains an \(H\)-tiling covering at least \((x-\eta)n\) vertices.
\end{thm}

\begin{proof}
Define \(H^*\) as in the proof of Theorem \ref{generalalmostmain}. Applying Theorem \ref{xdegseqKomlos} with input \(H^*\),
we obtain that \(G\) contains an \(H^*\)-tiling covering all but at most \((x-\eta) n\) vertices. Since each copy of \(H^*\) has a perfect \(H\)-tiling, we conclude that \(G\) contains an \(H\)-tiling covering all but at most \((x-\eta)n\) vertices.
\end{proof}

\section*{Acknowledgements}
The third author would like to thank Jan Hladk\'y and Diana Piguet for suggesting the question studied in this paper.
We are also grateful to the referees for their helpful and careful reviews.

{\footnotesize \obeylines \parindent=0pt
\begin{tabular}{lll}
	Joseph Hyde \& Andrew Treglown  &\ &  Hong Liu\\
	School of Mathematics 					&\ & Mathematics Institute		\\
	University of Birmingham				&\ & University of Warwick 	\\
	Birmingham											&\ & Coventry \\
	B15 2TT													&\ & CV4 2AL \\
	UK															&\ & UK
\end{tabular}
}
\begin{flushleft}
{\emph{E-mail addresses}:
\tt{jfh337@bham.ac.uk, a.c.treglown@bham.ac.uk, h.liu.9@warwick.ac.uk}}
\end{flushleft}

\end{document}